\documentclass[12pt,reqno]{amsart}
\usepackage[english]{babel}
\usepackage[applemac]{inputenc}
\usepackage[T1]{fontenc}
\usepackage{wrapfig}

\usepackage[top=25mm, bottom=20mm, left=35mm, right=35mm]{geometry}

\usepackage{bbm}
\usepackage{float, verbatim}
\usepackage{amsmath}
\usepackage{amssymb}
\usepackage{amsthm}
\usepackage{amsfonts}
\usepackage{graphicx}
\usepackage{mathtools}
\usepackage[pagebackref]{hyperref}
\usepackage{color}
\usepackage{tikz}
\usepackage{xcolor}
\usepackage{pgfplots}
\usepackage{caption}
\usepackage{subcaption}
\usepackage{enumerate}
\usepackage{url}

\hypersetup{
	colorlinks=true,
	linkcolor=blue,
	filecolor=magenta,      
	urlcolor=cyan,
	citecolor= red,
}

\pgfplotsset{compat=1.18}

\newtheorem*{thm*}{Theorem}
\newtheorem*{conj*}{Conjecture}
\newtheorem*{ques*}{Question}
\newtheorem*{rem*}{Remark}
\newtheorem*{defn*}{Definition}
\newtheorem*{mainques*}{Main questions}

\newtheorem{thm}{Theorem}[section]

\newtheorem{lem}[thm]{Lemma}
\newtheorem{cor}[thm]{Corollary}

\newtheorem{conj}[thm]{Conjecture}
\newtheorem{rem}[thm]{Remark}

\numberwithin{equation}{section}

\def \Leb {{\mathrm{Leb}}}

\def \bxi {{\boldsymbol{\xi}}}

\def \mmod {{\: \mathrm{mod}\:}}

\def \one{\mathbf{1}}

\def \bN {{\mathbb N}}
\def \bQ {{\mathbb Q}}
\def \bR {{\mathbb R}}
\def \bZ {{\mathbb Z}}

\def \cS {{\mathcal{S}}}

\def \alp {{\alpha}}
\def \bet {{\beta}}
\def \gam {{\gamma}}
\def \del {{\delta}}

\def \epsilon {{\varepsilon}}
\def \lam {{\lambda}}
\def \sig {{\sigma}}
\def \tet {{\theta}}
\def \ome {{\omega}}

\def \Del {{\Delta}}
\def \Lam {{\Lambda}}
\def \Tet {{\Theta}}
\def \Ome {{\Omega}}

\def \bt {{\mathbf{t}}}

\def \bzero {{\boldsymbol{0}}}

\def \d {{\mathrm{d}}}

\def \bE {\mathbb E}

\def \bN {\mathbb N}

\def \bQ {\mathbb Q}
\def \bR {\mathbb R}

\def \bZ {\mathbb Z}

\def \bT {\mathbb T}

\def \ba {\mathbf a}
\def \bb {\mathbf b}
\def \bc {\mathbf c}

\def \be {\mathbf e}

\def \bj {\mathbf j}

\def \bm {\mathbf m}
\def \bn {\mathbf n}

\def \bp {\mathbf p}

\def \bs {\mathbf s}
\def \bt {\mathbf t}
\def \bu {\mathbf u}
\def \bv {\mathbf v}
\def \bx {\mathbf x}

\def \by {\mathbf y}
\def \bz {\mathbf z}

\def \bzero {\mathbf 0}

\def \bxi {{\boldsymbol{\xi}}}

\def \cA {\mathcal A}

\def \cD {\mathcal D}
\def \cE {\mathcal E}

\def \cH {\mathcal H}
\def \cI {\mathcal I}
\def \cJ {\mathcal J}
\def \cK {\mathcal K}

\def \cN {\mathcal N}

\def \cP {\mathcal P}

\def \cS {\mathcal S}

\def \cW {\mathcal W}

\def \le {\leqslant}
\def \leq {\leqslant}
\def \ge {\geqslant}
\def \geq {\geqslant}

\def \dim {\mathrm{dim}}
\def \dimh {{\mathrm{\dim_H}}}

\def \ord {\mathrm{ord}}

\def \lcm {\mathrm{lcm}}

\def \dag {\dagger}

\def \d {{\mathrm{d}}}
\def \Leb {{\mathrm{Leb}}}

\def \ds1 {\mathds{1}}

\def \alp {{\alpha}}
\def \bet {{\beta}}
\def \gam {{\gamma}}
\def \del {{\delta}}

\def \lam {{\lambda}}
\def \ome {{\omega}}

\title[On a conjecture of Velani]{On a conjecture of Velani}

\author{Sam Chow}
\address{Sam Chow, Mathematics Institute, Zeeman Building, University of Warwick, Coventry CV4 7AL, UK}
\curraddr{}
\email{Sam.Chow@warwick.ac.uk}

\author{Han Yu}
\address{Han Yu, College of Mathematics and Statistics, Center of Mathematics, Chongqing University, Chongqing, 401331, China}
\curraddr{}
\email{han.yu.2@cqu.edu.cn}

\thanks{}



\begin{document}

\makeatletter
\providecommand\@dotsep{5}
\makeatother

\begin{abstract}
We investigate the approximation of elements of the middle-third Cantor set by dyadic rationals, and natural extensions. We prove a conjecture of Velani, completing the metric theory for the Cantor measure. Our key method enables a quantitative study of the incommensurability of binary and ternary expansions of integers, initiated 
by Furstenberg. The method has other applications, e.g. $\sum_{k}|\hat{\mu}(2^k)|^\sigma<\infty$ for any $\sig>0$, where $\mu$ is the Cantor measure.
\end{abstract}
\maketitle
\section{Introduction}

\thispagestyle{empty}

\subsection{Mahler, Velani, and Furstenberg}

Throughout, let 
\[
\psi: \bN \to [0,1]
\]
be a non-increasing function. We consider the set
\[
\cW(\psi) = \{ x \in \bR: \exists^\infty n \in \bN \quad
\| n x \| < \psi(n) \}
\]
of $\psi$-well approximable numbers, where $\| \cdot \|$ denotes the distance to $\bZ$.
In \emph{Some suggestions for further research} \cite{Mah1984}, Mahler asked how closely elements of the middle-third Cantor set $\cK$ can be approximated by rationals. There has been recent, dramatic progress on the subject of diophantine approximation on fractals via dynamical methods by Khalil--Luethi \cite{KL2023}, and ultimately by B\'enard--He--Zhang \cite{BHZ2026, BHZ}. In particular, the generic approximation rate has been pinpointed, for the Cantor set and for more general fractals.

\begin{thm} [B\'enard--He--Zhang]
Let $\mu$ be the Cantor measure. Then
\[
\mu(\cW(\psi)) = \begin{cases}
0, &\text{if } \sum_{n=1}^\infty \psi(n) < \infty \\
1, &\text{if }
\sum_{n=1}^\infty \psi(n) = \infty.
\end{cases}
\]
\end{thm}
\noindent
The Cantor measure is the canonical probability measure on $\cK$. It coincides with the restriction of the Hausdorff $\gam$-measure to $\cK$, for
\[
\gam = \frac{\log 2}{\log 3} = \dimh(\cK).
\]

An alternative approach to these problems uses Fourier analysis \cite{CY2, Yu}. Such tools are especially robust and have seen other applications relating to diophantine approximation on fractals \cite{ABCY2023, ACY2023, CVY2026, CY2, DJ}. Moreover, we see in the present work that Fourier analysis plays a vital role in our full solution to a conjecture of Velani.

Levesley, Salp, and Velani \cite{LSV2007} were able to find very good approximations by using triadic rationals. For any integer $p \ge 2$, let
\[
\cW_p(\psi) = \{ x \in \bR: \exists^\infty k \in \bN \quad \| p^k x \| < \psi(k) \}.
\]
\begin{thm} [Levesley--Salp--Velani]
For $s \ge 0$ and $\gam = \frac{\log 2}{\log3}$,
\[
\cH^s(\cW_3(\psi) \cap \cK) =
\begin{cases}
0, &\text{if } \sum_{k=1}^\infty \psi(k)^s (3^k)^{\gam - s} < \infty \\
\cH^s(\cK), &\text{if }
\sum_{k=1}^\infty \psi(k)^s (3^k)^{\gam - s} = \infty.
\end{cases}
\]
\end{thm}

We note that the case $s = \gam$ yields, for $\mu$ being the Cantor measure,
\[
\mu(\cW_3(\psi)) =
\begin{cases}
0, &\text{if } \sum_{k=1}^\infty \psi(k)^\gam < \infty \\
1, &\text{if }
\sum_{k=1}^\infty \psi(k)^\gam = \infty.
\end{cases}
\]
The symmetry of $\cK$ under the `times three' action on $\bT$ suggests that this behaviour is rather special. The situation is different in the case $p = 2$ of dyadic approximation \cite{ABCY2023, ACY2023, Bak2025}. Velani made the following more precise prediction.

\begin{conj} [Velani]
Let $\mu$ be the Cantor measure. Then
\[
\mu(\cW_2(\psi)) =
\begin{cases}
0, &\text{if } \sum_{k=1}^\infty \psi(k) < \infty \\
1, &\text{if }
\sum_{k=1}^\infty \psi(k) = \infty.
\end{cases}
\]
\end{conj}

This is analogous to Khintchine's theorem \cite{Khi1924}, and is a weak form of independence for the `times two' and `times three' actions --- a phenomenon observed by Furstenberg \cite{Fur1967, Fur1970} in the 1960s. In fact, we saw in \cite{ACY2023} --- and we also find in the present work --- that this type of independence is felt even more strongly on the Fourier side.

For integers $p \ge 2$ and $d \ge 1$, let
\[
\cW_p(\psi, d) = \{ \bx \in \bR^d: \exists^\infty k \in \bN \quad
\| p^k \bx \| < \psi(k)
\}.
\]
The purpose of this work is to establish the following generalisation of Velani's conjecture, the terminology for which will be explained in \S \ref{prep}.

\begin{thm}
\label{MainThm}
Let $p \ge 2$ be an integer, and let $q \ge 3$ be a prime such that $q \nmid p$. Let $\mu$ be a missing-digit measure corresponding to an affinely irreducible missing-digit set $\cK \subset [0,1]^d$ with base $q$. Then
\[
\mu(\cW_p(\psi, d)) =
\begin{cases}
0, &\text{if } \sum_{k=1}^\infty \psi(k)^d < \infty \\
1, &\text{if }
\sum_{k=1}^\infty \psi(k)^d = \infty.
\end{cases}
\]
\end{thm}

\begin{rem}
It may seem that this is a special case of B\'{e}nard--He--Zhang's theorem, by supporting the approximating function on powers of $p$ integers. It is not. Indeed, the monotonicity of the approximating function therein cannot be dropped, as can be seen from Levesley--Salp--Velani's theorem.
\end{rem}

\subsection{Existing and new methods}

To fix ideas, we assume here that $d = 1$. This is the most difficult case.

For the convergence theory, we apply the first Borel--Cantelli lemma. We also apply Parseval's identity, to pass to Fourier space $\bZ$. As in \cite{ABCY2023, ACY2023, Bak2025, Ban, CVY2026, DLWW, Yu}, we exploit the classical product formula for the Fourier coefficients of $\mu$, namely
\begin{equation}
\label{ProductFormula}
|\hat \mu(\xi)| = \prod_{j=1}^\infty g(\xi/q^j).
\end{equation}
If $\mu$ has associated probability vector $\bp \in [0,1]^q$, then $g$ is given by
\[
g(y) = \left| \sum_{j=0}^{q-1} p_j e(jy) \right|,
\]
where $e(z) = e^{2 \pi i z}$.

By the arithmetic of our problem and the uncertainty principle, the important frequencies to consider have the form
\[
\xi = p^k t,
\]
where $k$ and $t$ are at roughly the same scale $Q$, and we may assume that $t \in \bN$. In particular, now $\xi$ has relatively few digit changes in base $p$. Since and $p$ and $q$ are multiplicatively independent, groundbreaking work of Stewart \cite{Ste1980} forces $\xi$ to have many digit changes in base $q$. Then \eqref{ProductFormula} entails Fourier decay.
All of the above was observed in \cite{ACY2023}, delivering progress towards Velani's conjecture. Further progress was made in \cite{Bak2025} and then in \cite{ABCY2023} by observing that, given $t$, the quantity $|\hat \mu(p^k t)|$ decays polynomially away from a polynomially sparse set of values of $k$. 

In the present work, we closely examine the digit strings and their evolution under multiplication by $p$. We start by introducing a qualification process for $k$, wherein $k$ qualifies if $p^k t$ has many stretches of digits in base $q$ that lack digit changes. Unqualified frequencies have significant Fourier decay, so we only need to consider qualified values of $k$.

We are allowed to work with a more general notion of digit changes, as the only way to avoid Fourier decay is to have long strings of digits $0$ or $q-1$ in base $q$. By symmetry, we may assume that if $k$ is qualified then $p^k t$ has many long strings of zeros. Each time this happens, we look at a carefully-chosen number $T$ of preceding digits. It follows from the structure of the group $(\bZ/q^T \bZ)^\times$ that the values of $p^{k+\Del} t$ modulo $q^T$ are `almost' distinct for $\Del = 1, 2, \ldots, q^T$. An observation going back to Erd\H{o}s~\cite{Erd1940} and Kahane \cite{Kah1971} then ensures that, save for a small exceptional set of $\Del$, there are many digit changes in this prior string. This local data can be concatenated into a global count for the number of digit changes, giving sufficient average Fourier decay.

\bigskip

This framework also handles much of the analysis for the divergence theory, but some additional arguments are needed. By the divergence Borel--Cantelli lemma~\cite{BV2023}, we need to consider the mean overlap of pairs of events, i.e. averages of
$
\mu(\cA_{k_1} \cap \cA_{k_2}),
$
where
\[
\cA_k = \{ x \in [0,1]: 
\| p^k x \| < \psi(k) \}.
\]
Applying Parseval then leads us to estimate quantities of the form
\[
\sum_{k_1, k_2, t_1, t_2}
|\hat \mu(p^{k_1} t_1 + p^{k_2} t_2)|.
\]
The summands where $p^{k_1} t_1 + p^{k_2} t_2=0$ correspond to the Lebesgue second moment, and we need to bound the other contributions, in the spirit of the `moment transference principles' framework of \cite{CY1}.
A version of the argument used for the convergence theory prevails in the regime where $k_1 \approx k_2$, but we need another approach to handle the complementary regime.

Now suppose that we fix $k_1 \asymp Q$ and average over $k_2 \asymp Q$ with $k_2 - k_1$ somewhat large. The following steps are complicated and rely on two essential features:
\begin{itemize}
\item The set $\cA_k$ can be assumed to consist of $p$-boxes, which loosely means that the vertices and side lengths are integer multiples of an integer power of $p$.
\item The measure $\mu$ is self-similar \cite{BP}. Roughly speaking, this means that $\mu$ looks the same when you `zoom in' on a branch of the missing-digit set at any level of the iterative construction.
\end{itemize}
After a series of linear transformations, and upon using the relationship between Fourier series and Fourier transforms, we are led to bound averages of the form
\[
\sum_{k_1 + g(Q) \le k_2 \ll Q} |\hat \mu(p^{k_2} t / q^\ell)|,
\]
for some suitably chosen $g(Q) = o(Q)$ and $\ell \in \bN$.
Although the frequency $p^{k_2} t / q^\ell$ is not an integer, we still have the product formula \eqref{ProductFormula}, and essentially the same digit structure in base $q$. The constraint $k_2 \ge k_1 + g(Q)$ ensures that sufficiently many digits remain for our arguments to carry through.

Finally, to bootstrap from positive measure to full measure, we leverage the ergodicity of Bernoulli schemes to establish an abstract zero-full law in Lemma~\ref{ZeroFull}. Ours differs from classical zero-full laws \cite{Cas1950, Gal1961}, in that it applies to all approximating functions at once, as opposed to any single approximating function.

\subsection{Average Fourier decay along a lacunary sequence}

Our main technical result is powerful, and may find utility elsewhere. It is convenient to begin with a special case. We present a more general version in \S \ref{tec}. We will discuss $(q,m)$-good functions in \S \ref{prep}. For now, note that any missing-digit measure in one ambient dimension has $(q,m)$-good Fourier transform for some $m \in \bN$, where $q$ is the base.

\begin{thm}
\label{AvgDecay}
Let $p \ge 2$ be an integer, and let $q \ge 3$ be a prime such that $q \nmid p$. Let $f$ be a $(q,m)$-good function, for some $m \in \bN$, and fix $c > 0$ sufficiently small in terms of $p,q,f$. Let $Q \ge 1$, and let $t \in \bN$ with
$
t \le Q^2.
$
Then
\[
\sum_{Q < k \le 2Q} f(p^k t) \ll_{p,q,m,f}
Q^{-c/\log \log Q}.
\]
\end{thm}

By summing the right hand side over powers of two, we discover the following pleasant consequence.

\begin{cor}
\label{SubseqDecay}
Let $p \ge 2$ be an integer, and let $q \ge 3$ be a prime such that $q \nmid p$. Let $\mu$ be a missing-digit measure corresponding to a missing-digit set $\cK \subset [0,1]$ with base $q$, and let $t \ne 0$ be an integer. Then, for any $\sig > 0$,
\[
\sum_{k \in \bN} 
|\hat \mu(p^k t)|^\sig < \infty.
\]
\end{cor}

Theorem \ref{AvgDecay} also leads to pointwise effective equidistribution of $\times p$-orbits sampling from $\mu$, complementing \cite{DGW24} and answering a question of Algom  \cite[Problem 4.1]{Alg} for the stated cases. The results in \cite{DGW24} concern self-similar systems whose translation structure is not `rational'. Equidistribution holds ineffectively for any expanding map coming from an integer matrix. In our case, the self-similar systems are highly `rational' and there are also conditions on the expanding factor and the contraction ratio of the underlying measure.

\begin{thm}
\label{thm: QN_intro}
Let $f\in C^2(\bT)$ and $\del > 0$. Let $\mu$ be a missing-digit measure with prime base $q$, and let $p \ge 2$ be coprime to $q$. Then, for $\mu$-a.e. $x$,
\[
\left| \frac{1}{T} \sum_{k=0}^{T-1} f(p^k x) - \int_{\bT} f(x) \d x \right| \ll_{f,\del,x} T^{\del - 1/2}
\qquad (T \ge 1).
\]
\end{thm}
The exponent is optimal in the case where $\mu$ is Lebesgue measure.
This is due to the connection between $\times p$ sequences on $\bT$ and i.i.d Lebesgue random sequences. The exponent $1/2$ can be obtained via the central limit theorem \cite{Kac46,Phi75}. Theorem~\ref{thm: QN_intro} can be regarded as a quantitative version of Cassels' classical result on the $2$-normality of the Cantor measure \cite{Cas1959}.

\subsection{Open problems}

\subsubsection{The assumption that the base is prime}

The fact that $q$ is prime in all the above results is crucial in our proofs. However, it is unlikely to be necessary. In fact, we believe that all of the above results hold under a weaker assumption that $\log p/\log q\not\in\bQ$, where $p,q \ge 2$ are integers.

\subsubsection{General self-similar measures}
\label{sec: open question general SS}

We know from the aforementioned work of Levesley--Salp--Velani \cite{LSV2007} that Theorem \ref{MainThm} fails in the case $p=q=3$. Nonetheless, the analogue of Velani's conjecture should hold in some generality for other measures $\mu$. For example, Pollington--Velani--Zafeiropoulos--Zorin \cite{PVZZ2022} showed that if $\hat{\mu}(\xi) \ll (\log |\xi|)^{-\sig}$ for some sufficiently large $\sig>0$, then the analogue of Theorem \ref{MainThm} holds with this choice of $\mu$. This can be described as \emph{lacunary diophantine approximation}. A natural task is to establish the lacunary, Khintchine-type dichotomy for general classes of self-similar measures without Fourier decay. See \cite{Bre21, Rap22, VY22} for examples such measures.

\subsubsection{Dropping the monotonicity condition} 

We require the monotonicity of the approximation function $\psi$ for both the convergence and the divergence theory. However, Velani's conjecture might hold without this assumption.
When $\mu$ is Lebesgue measure, relaxing monotonicity is a consequence of a theorem of Duffin and Schaeffer \cite{DS41}, see also \cite{KM20}. More generally, the monotonicity is not needed if $\mu$ has some mild Fourier decay \cite{PVZZ2022}.  Given the nature of our proof we suspect that dropping  monotonicity, if possible, would be delicate and require new ideas.

\subsubsection{A central limit theorem}\label{ques: CLT}

Let $p,q \ge 2$ be integers such that $\log p/\log q\not\in \bQ$. Given $x\sim \mu$ on $\bT$, where $\mu$ is a missing-digit measure in base $q$, we wish to understand the $\times p$ orbit of $x\in\bT$. It is unclear whether a central limit theorem holds in this case. More precisely, for $T>0$, consider the Ces\`{a}ro average
\[
A^f_T(x)=\frac{1}{T}\sum_{k=0}^{T-1} f(p^k x),
\]
for $f\in C^1(\bT)$. If we consider $f(x)$ with $x \sim \Leb(\bT)$, then $\bE(f) = \int_\bT f(x) \d x$. Let $\sig^2_f$ be the Green--Kubo $\times p$ variance, given by
\begin{align*}
\sig_f^2 & = \lim_{T\to\infty} T\bE((A^f_T-\bE(f))^2)\\
 & =\int_\bT (f(x)-\bE(f))^2 \d x + 2\sum_{k\geq 1} \int_\bT (f(x)-\bE(f)) (f(p^k x) - \bE(f)) \d x.
\end{align*}
Convergence follows from the Fourier decay of $C^1$ functions and \cite[Lemma 1]{Kac46}. We also assume that $f$ satisfies $\sig_f := \sqrt{\sig_f^2} \in (0, \infty)$. We then normalise $A^f_T$ via
\[
Z^f_T(x) = \frac{A^f_T(x)-\bE(f)}{\sig_f/\sqrt{T}}.
\]
If $x \sim \Leb(\bT)$, then $Z^f_T\to \cN(0,1)$ in distribution \cite{Kac46, Phi75}. We do not know whether this holds for $x\sim \mu$, despite strong numerical evidence (see \S \ref{sec: Exp}).

\subsection*{Notation}
Given a set $\cS$, we write $\one_\cS$ for its indicator function.
For a real vector $\bx \in \bR^d$, we write $|\bx|$ for its supremum norm, and $\| \bx \|$ for its supremum distance to $\bZ^d$.

For complex-valued functions $f,g$, we write $f \ll g$ or $f = O(g)$ or $g = \Ome(f)$ if there exists a constant $C$ such that $|f| \le C|g|$ pointwise. We use a subscript in this notation to indicate if the implied constant $C$ may depend on specific parameters.

\subsection*{Acknowledgements}

We thank Sanju Velani for suggesting the problem many years ago. We also thank Demi Allen and Simon Baker, with whom we wrote earlier foundational papers \cite{ABCY2023, ACY2023} containing essential techniques for this work. Finally, we acknowledge a regional Canley-Birmingham train from the West Midlands Railway/London Northwest Railway. That train has alternating full-empty carriages, which provides the key idea for proving Theorem \ref{AvgDecay}, the main estimate in this paper.

\subsection*{Funding}

HY was supported by the Leverhulme Trust (ECF-2023-186).

\subsection*{Rights}

For the purpose of open access, the authors have applied a Creative Commons Attribution (CC-BY) licence to any Author Accepted Manuscript version arising from this submission.

\subsection*{AI declaration}

The first online version of this work, posted to arXiv, was completed entirely without AI, save for the following exceptions.
Figure~\ref{GoodFigure}  was produced with the assistance of GPT-5.6 Sol. Its code was generated using data provided by the authors through traditional coding in Python. Several Google search results  were provided within Gemini chats. 

GPT-6 Astra was then used to review the first online version and provided many challenges, including a gap in the original proof of Theorem \ref{thm: QN_intro} for a missing maximal inequality, and in Section \ref{sec: div} for non-doubling measures. The revisions were made by the authors and subsequently checked by GPT-6 Astra.

\section{Preparation}
\label{prep}

\subsection{Fourier analysis}

Fourier transforms and Fourier series are discussed in many textbooks, including \cite{Gra2014}. The \emph{Fourier transform} of $f \in L^1(\bR^d)$ is
\[
\hat f(\bt) = 
\int_{\bR^d} f(\bx) e(-\bt \cdot \bx) \d \bx \qquad (\bt \in \bR^d).
\]
For $F \in L^1((\bR/\bZ)^d)$, the \emph{Fourier coefficients} of $F$ are given by
\[
\hat F(\bn) = 
\int_{[0,1]^d} F(\bx) 
e(-\bn \cdot \bx) \d \bx 
\qquad (\bn \in \bZ^d).
\]
Now suppose $f \in L^1(\bR^d)$ is compactly supported, and let us define $F \in L^1((\bR/\bZ)^d)$ by
\[
F(\bx) = \sum_{\bn \in \bZ^d} f(\bx+\bn).
\]
Then, for $\bn \in \bZ^d$,
\begin{align}
\notag
\hat F(\bn) &=
\int_{[0,1]^d}  F(\bx) e(-\bn \cdot \bx) \d \bx
= \int_{[0,1]^d}  \sum_{\bm \in \bZ^d} f(\bx+\bm) e(-\bn \cdot \bx) \d \bx \\
\notag
&= \sum_{\bm \in \bZ^d} \int_{[0,1]^d} f(\bx + \bm) e(-\bn \cdot (\bx + \bm)) \d \bx \\
\label{coincide}
&= \int_{\bR^d} f(\bx) e(-\bn \cdot \bx) \d \bx = \hat f(\bn).
\end{align}
To facilitate the Fourier analysis, we will use a non-zero bump function 
\[
\ome: \bR^d \to [0,1]
\]
supported on $[-2,2]^d$, and
majorising or minorising $\one_{[-1,1]^d}$. We will adapt $w$ to a suitable scale, defining
\[
\ome_\del(\bx) = \ome(\bx/\del)
\]
for $\del > 0$ and $\bx \in \bR^d$. Since bump functions are Schwartz, so too are their Fourier transforms, whence
$
\hat \ome(\bn) \ll_N (1 + |\bn|)^{-N}
$
for any $N \in \bR$, and more generally
\[
\hat \ome_\del(\bn) \ll_N \del^d (1 + \del |\bn|)^{-N}
\qquad (\bn \in \bR^d).
\]

\subsection{Missing-digit sets and measures}

Let $d \ge 1$ and $q \ge 3$ be integers, and let $\cD$ be a subset of $\{ 0, 1, \ldots, q-1 \}^d$ with $\# \cD \ge 2$.
We associate to this data the \emph{missing-digit set}
\[
\cK = \cK_{q,\cD} = \left \{
\sum_{j=1}^\infty \ba^{(j)} q^{-j}: \ba^{(j)} \in \cD \ \forall j \right \}.
\]
Given $\cK$ as above and a probability vector $\bp$ of length $q^d$ fully supported on $\cD$, we associate the \emph{missing-digit measure} $\mu$. This is the law of the random series $\sum_{j=1}^\infty \ba^{(j)} q^{-j}$, where the $\ba^{(j)}$ are i.i.d. with law $\bp$.

We say that a missing-digit measure $\mu$ is \emph{arithmetically irreducible} if 
\[
\bx \cdot \ba \equiv \bx \cdot \bb \mmod \bZ \ \forall \ba\neq \bb \in \cD 
\quad
\Longrightarrow
\quad
\bx \in \bZ^d,
\]
and \emph{affinely irreducible} if it is not supported on a proper affine subspace. Note that if $\mu$ is arithmetically irreducible then it is affinely irreducible.

\subsection{Tools to find and exploit digit changes}

We retain the notation of the previous subsection, but assume that $d = 1$. Write $n \in \bN$ in base $q$ as
\[
n = \sum_{j=0}^N b_j q^{N-j},
\]
where $b_j \in \{ 0 , 1, \ldots, q-1 \}$ for all $i$ and $b_0 \ne 0$. The expansion has a \emph{generalised digit change} in position $j$ if $b_j \notin \{ 0, q - 1 \}$ or $b_j \ne b_{j+1}$. It has a \emph{digit change} in position $j$ if $b_j \ne b_{j+1}$.

A function $f: \bN \to [0,1]$ is \emph{$q$-good} if there exists a positive constant 
$
c = c(f)
$
such that if $n \in \bN$ has $T$ many generalised digit changes then
\[
f(n) \ll e^{-cT}.
\]
Clearly, the Lebesgue measure is $q$-good, for any $q$.
\begin{lem} [Fourier decay from digit changes]
\label{qgood}
Suppose $\mu$ is arithmetically irreducible.
Then $|\hat \mu|$ is $q$-good.
\end{lem}

\begin{proof}
Let $(p_0, \ldots, p_{q-1})$ be the probability vector defining $\mu$, and recall \eqref{ProductFormula}. As $y$ is arithmetically irreducible,
\[
\| y \| = 0
\quad
\Longleftrightarrow
\quad
g(y) = 1.
\]
Thus, by continuity and compactness, there exists $c = c(\mu) > 0$ such that
\[
\| y \| \ge 1/q^2 
\quad
\Longrightarrow 
\quad
g(y) \le e^{-c}.
\]
It follows that $\hat \mu$ is $q$-good.
\end{proof}

\bigskip

It is important for us that if a positive integer is well-structured in base $p$ then almost every `Stewart interval' contains a digit change in base $q$.

\begin{thm} [Stewart]
\label{Stewart}
Let $p,q \in \bN$ be multiplicatively independent. Let $C$ be a positive constant, sufficiently large in terms of $p$ and $q$. Let $n \in \bN$ be a positive integer, sufficiently large in terms of $p,q,C$. Let $r$ and $s$ be the number of digits of $n$ in bases $p$ and $q$, respectively, and let
\[
\frac13 C \log r \le \tet \le 3C \log r.
\]
Let $j \in \bN$ with $\tet^j \le \min \{r,s\}$. Then the interval 
\[
\Tet_j = [\tet^{j-1},\tet^j]
\]
contains a digit change of $n$ in base $p$ or base $q$.
\end{thm}

\begin{proof}
By the argument of \cite[\S 3]{Ste1980}, if $\alp,\bet$ are digits in base $p,q,$ respectively, then the interval $\Tet_j$ contains either a non-$\alp$ digit in base $p$ or a non-$\bet$ digit in base $q$. This holds for all $\alp, \bet$, and we necessarily have $p,q \ge 2$, so the result follows.
\end{proof}

\bigskip

We need some modular arithmetic to understand how the $\times p$ action affects digit strings in base $q$. For coprime $p,q \in \bN$ and for $\ell \in \bN$, let $H_\ell(p)$ be the order of $p$ modulo $q^\ell$, and define
\[
\mathrm{occ}_{q} p=\lim_{\ell\to\infty} \frac{H_{\ell}(p)}{q^\ell}.
\]
We show that the limit exists and is positive, restricting for simplicity to the case where $q$ is prime.

\begin{lem} [Local $q$-occupancy]
\label{lma: local occ}
Let $q$ be prime, and let $p \ge 2$ with $q \nmid p$. Then $\mathrm{occ}_{q} p>0$.
\end{lem}

\begin{proof} Put $u = H_1(p)$ if $q$ is odd and let $t \in \bN$ be maximal for which
$
p^{u} \equiv 1 \mmod q^t.
$
If $q=2$, then $u=H_2(p)$, and $t\geq 2$ is  maximal for which 
$
p^u\equiv 1\mmod q^t.
$

We show by induction that
\begin{equation}
\label{Hell}
H_\ell(p) = u q^{\ell - t} \qquad (\ell \ge t).
\end{equation}
This will imply that $\mathrm{occ}_q p=up^{-t}>0$.

Since $H_1(p) = u$, we also have $H_t(p) = u$, confirming the base case. Now let $\ell \ge t + 1$, and assume that $H_{\ell-1}(p) = u q^{\ell-1-t}$. Then
\[
p^{u q^{\ell-2-t}} = 1 + kq^{\ell-1},
\]
for some $k \in \bZ$, so $p^{u q^{\ell-t}} \equiv 1 \mmod q^\ell$, whence
\[
H_\ell(p) \mid u q^{\ell-t}.
\]
We also infer from the inductive hypothesis that
\[
u q^{\ell - 1 - t} \mid H_\ell(p).
\]

Assume for a contradiction that $H_\ell(p) \ne u_1 q^{\ell-t}$. Since $q$ is prime, we must have
$H_\ell(p) = u q^{\ell - 1 - t}$, so
\[
p^{u q^{\ell - 1 - t}} \equiv 1 \mmod q^\ell.
\]
For some $b \in \bZ$ coprime to $q$, we have $p^{u} = 1 + bq^t$, so
\[
p^{u q^{\ell - 1 - t}}  = (1 + bq^t)^{q^{\ell-1-t}} \equiv 1 + bq^{\ell-1} \mmod q^\ell.
\]
This contradicts the previous congruence, closing the induction and implying that $\mathrm{occ}_q p>0$.
\end{proof}

\begin{lem} [Action of $\times p$ on $q$-strings I]
Let $q$ be prime, let $p \ge 2$ with $q \nmid p$, and let $x \in \bN$ have $q$-expansion $\langle x \rangle$. Let $T \in \bN$ and suppose
\[
\langle x \rangle
= \langle y \rangle
0 \ldots 0 \langle z \rangle,
\]
where $q \nmid y$ and the string of zeros has length at least
$
2H_T(p) \frac{\log p}{\log q}.
$
Then, for $k = 0, 1, \ldots, H_T(p) - 1$, we have
\[
\langle p^k x \rangle = 
* \langle p^k y \mmod q^T \rangle 0 \ldots 0 \langle p^k z \rangle,
\]
where the string of zeros is non-trivial and $*$ is a placeholder for an arbitrary string. Moreover, the string 
\[
\langle p^k y \mmod q^T \rangle
\]
realises exactly $H_T(p)$ many possibilities as $k$ varies.
\end{lem}

\begin{proof} 
For the form of the expansion and the non-triviality of the string of zeros, note that
\[
\max \left \{ 1,
\frac{\log(p^k)}{\log q}
\right \}
< \frac{H_T(p) \log p}{\log q},
\]
so some zeros will remain after multiplying $z$ by $p^k$. The final claim is clear because all of these residues are different.
\end{proof}

The result naturally extends to concatenations of strings, as follows.

\begin{lem} [Action of $\times p$ on $q$-strings II]
\label{Act2}
Let $q$ be prime, and let $p \ge 2$ with $q \nmid p$.
Let $T, x \in \bN$ and suppose
\[
\langle x \rangle =
\langle y_1 \rangle 0 \ldots 0 \langle y_2 \rangle 0 \ldots 0 \langle y_\ell \rangle 0 \ldots 0 \langle z \rangle,
\]
where $q \nmid y_1 \cdots y_\ell$, and where each string of zeros has length at least $2H_T(p) \frac{\log p}{\log q}$. Then, for $k=0,1,\ldots,H_T(p)-1$, we have
\begin{align*}
\langle p^k x \rangle &=
* \langle p^k y_1 \mmod q^T \rangle 0 \ldots 0 *
\langle p^k y_2 \mmod q^T \rangle 0 \ldots 0 \\
&\qquad *
\langle p^k y_\ell \mmod q^T \rangle 0 \ldots 0 \langle p^k z \rangle,
\end{align*}
where each string of zeros is non-trivial, and each string $\langle p^k y_j \mmod q^T \rangle$ realises exactly $H_T(p)$ many possibilities as $k$ varies.
\end{lem}

We record here an easy arithmetic trick.

\begin{lem}
[Primary school arithmetic trick] 
\label{FlipDigits}
Let $q$ be prime.
Suppose $x \in \bN$ has the $q$-digit structure
\[
\langle x \rangle = \langle y \rangle q-1 \dots q-1,
\]
namely $\langle y \rangle$ followed by $L > 0$ many digits $q-1$. 
\begin{enumerate}[(a)]
\item
Let $n \in \bN$ with $n < q^{L}$. Then
\[
\langle n x \rangle = \langle n(y+1)-1 \rangle *.
\]
\item Let $p \ge 2$ with $q \nmid p$. Suppose that the last $q$-digit of $y$ is not $q-1$, and let $T \in \bN$ with $2H_T(p) \frac{\log p}{\log q} \leq  L$. Then the values of
\[
\lfloor p^kx/q^L \rfloor \mmod q^T
\qquad (0 \le k \le H_T(p)-1)
\]
are pairwise distinct.
\end{enumerate}
\end{lem}

\begin{rem}
In primary schools, at least the one that the second named author attended, integer subtraction involving carries is taught this way. For example, in decimal, $100-2=001+099-002=001+097=098=98.$
\end{rem}

\begin{proof}
\begin{enumerate}[(a)]
\item
As $x + 1 = (y+1)q^L$, we have
\[
nx = n(y+1)q^L - n
= (n(y+1)-1)q^L + (q^L-n).
\]
Since $n < q^L$, the conclusion follows.
\item As the last digit of $y$ is not $q-1$, we know that $y+1$ does not have the digit $0$ at the end. The middle digit string in
\[
*\langle p^k (y+1) \mmod q^T \rangle*
\]
is pairwise distinct for $k = 0, 1, \ldots, H_T(p) - 1$, so the same holds for
\[
*\langle p^k (y+1)-1 \mmod q^T \rangle*.
\]
\end{enumerate}
\end{proof}

This lemma can also be concatenated to provide a digit-flipped version of Lemma \ref{Act2} as a direct consequence.

\begin{lem} [Action of $\times p$ on $q$-strings III]
\label{Act3}
Let $q$ be prime, and let $p \ge 2$ with $q \nmid p$.
Let $T, x \in \bN$ and suppose
\[
\langle x \rangle =
\langle y_1 \rangle q-1 \ldots q-1
\langle y_2 \rangle q-1 \ldots q-1
\cdots
\langle y_\ell \rangle q-1 \ldots q-1 *,
\]
where $q \nmid (y_1+1) \cdots (y_\ell+1)$, and where
each string of digits $q-1$ has length at least
\[
2H_T(p) \frac{\log p}{\log q}.
\]
For each $j$, let $Y_j$ be the integer represented by
the full prefix of $\langle x\rangle$ ending with
$\langle y_j\rangle$.
Then, for $k=0,1,\ldots,H_T(p)-1$, we have
\begin{align*}
\langle p^k x \rangle &=
* \langle p^k (Y_1+1)-1 \mmod q^T \rangle
q-1 \ldots q-1 \\
&\qquad *
\langle p^k (Y_2+1)-1 \mmod q^T \rangle
q-1 \ldots q-1 \\
&\qquad *
\langle p^k (Y_\ell+1)-1 \mmod q^T \rangle
q-1 \ldots q-1 *,
\end{align*}
where each string of digits $q-1$ is non-trivial, and
each string
\[
\langle p^k (Y_j+1)-1 \mmod q^T \rangle
\]
realises exactly $H_T(p)$ many possibilities as
$k$ varies. 
\end{lem}

\begin{proof}
Fix $j$, let $L_j$ be the length of the $j^{\mathrm{th}}$
string of digits $q-1$, and let $M_j$ be the number
of digits following this string. We write
\[
x=q^{M_j}x_j+z_j,
\qquad 0\le z_j<q^{M_j},
\]
where
$
x_j=(Y_j+1)q^{L_j}-1.
$
In particular,
\[
\langle x_j\rangle
=\langle Y_j\rangle q-1\ldots q-1.
\]
Since $Y_j$ ends with the digits of $y_j$,
the hypothesis implies that
\[
q\nmid Y_j+1.
\]

For $0\le k<H_T(p)$, we have $p^k<q^{L_j}$.
Note that
\[
\left\lfloor\frac{p^kx_j}{q^{L_j}}\right\rfloor
=p^k(Y_j+1)-1.
\]
Now Lemma~\ref{FlipDigits}(b), applied to
$x_j$, shows that
\[
p^k(Y_j+1)-1\mmod q^T
\]
realises exactly $H_T(p)$ many possibilities
as $k$ varies over $0,1,\ldots,H_T(p)-1$.

It remains to confirm that a non-trivial string
of digits $q-1$ survives.
Observe that
\[
p^k x
=p^k(Y_j+1)q^{L_j + M_j}
- p^k(q^{M_j} -z_j).
\]
As
\[
0 < p^k(q^{M_j} -z_j)< q^{M_j+L_j},
\]
we obtain
\[
\left \lfloor
\frac{p^k x}{q^{L_j + M_j}}
\right \rfloor
=
\left \lfloor
\frac{p^kx_j}{q^{L_j}}
\right \rfloor
= p^k(Y_j+1) - 1
\]
and in particular
\[
\langle p^k x \rangle = \langle p^k(Y_j + 1) - 1 \rangle \langle q^{L_j + M_j} - p^k(q^{M_j} - z_j) \rangle.
\]

Put
\[
b_k = \left\lceil\frac{k\log p}{\log q}\right\rceil.
\]
Then
\[
0<p^k(q^{M_j}-z_j)\le q^{M_j+b_k}
\]
and
\begin{align*}
&q^{L_j + M_j} -p^k(q^{M_j}-z_j) \\
&=
(q^{L_j-b_k}-1) q^{M_j+b_k}
+
\bigl(q^{M_j+b_k} -p^k(q^{M_j}-z_j)\bigr).
\end{align*}
Consequently, at least $L_j-b_k$ many consecutive
digits $q-1$ follow the prefix
$\langle p^k(Y_j+1)-1 \rangle$.
Since $q^{L_j} > p^k$, this number is
positive.
\end{proof}

Finally, we record a useful observation that can be traced back to the works of Erd\H{o}s \cite{Erd1940} and Kahane \cite{Kah1971}. 

\begin{lem} [Erd\H{o}s--Kahane]
\label{EK}
Let $q,T \ge 2$ be integers, and suppose $1/T \le \eta \le 1/2$. Let $C_\eta(T)$ count $q$-strings of length $T$ with at most $\eta T$ many generalised digit changes. Then
\[
C_\eta(T) \le (q^T)^{2 \eta (1 + \log(6/\eta))}.
\]
\end{lem}

\begin{proof}
Any $q$-string counted by $C_\eta(T)$ contains at most 
\[
k = \lfloor \eta T \rfloor \in [\eta T / 2, \eta T]
\]
many digits that are not $0$ or $q-1$. There are at most 
$
q^k \binom{T}{k}
$
many ways to choose these digits are their positions.
Among the digits that are $0$ or $q-1$, there are at most $k$ many digit changes, i.e. $0$ followed by $q-1$ or vice versa. Thus,
\[
C_\eta(T) \le 2q^k \binom{T}{k}^2 \le 2q^k \left( \frac {Te}{k} \right)^{2k}
\le 2q^{\eta T} (6/\eta)^{2\eta T} \le (q^T)^{2 \eta (1 + \log(6/\eta))}.
\]
\end{proof}

\subsection{Using affine irreducibility}

Affinely irreducible measures behave similarly to arithmetically irreducible measures, and form a more general class. In this subsection, we introduce a variant of Lemma \ref{qgood}. Given integers $q \ge 2$ and $m \ge 1$, a function $f: \bR \to [0,1]$ is \emph{$(q,m)$-good} if there exists a constant $c = c(f) > 0$ such that if $\ell, n \in \bZ$ with $\ell \ge 0$ and $|mn|$ has $T$ many generalised digit changes in base $q$ not among the final $\ell$ many digits, then $f(n/q^\ell) \ll e^{-cT}$.

\begin{lem}
\label{decent}
Let $q \ge 3$ be an integer, and let $\mu$ be a missing-digit measure corresponding to an affinely irreducible missing-digit set $\cK \subset [0,1]^d$ with base $q$. Then $|\hat \mu|$ is $(q,m)$-good in each variable, for  some $m \in \bN$.
\end{lem}

\begin{proof}
Let $\cD = \{ \bu^{(1)}, \ldots, \bu^{(r)} \}$ be the set of allowed digits.
As $\mu$ is affinely irreducible, the vectors
\begin{equation}
\label{differences}
\bu^{(2)} - \bu^{(1)},
\ldots,
\bu^{(r)} - \bu^{(1)}
\end{equation}
form a set of full rank. Consequently, their $\bZ$-span is a full lattice $\Lam$ of some determinant $m \in \bN$.

Let $\bx \in \bR^n$, and suppose
\[
\bx \cdot \bu^{(i)} \in \bZ
\qquad (1 \le i \le r).
\]
Let $\bv^{(1)}, \ldots, \bv^{(d)}$ form a basis for $\Lam$. These must be $\bZ$-linear combinations of the vectors \eqref{differences}, so
\[
\bx \cdot \bv^{(j)} \in \bZ \qquad (1 \le j \le d).
\]
In particular, the vector $\bx$ lies in the dual lattice $\Lam^*$ which, by Cramer's rule, is contained in the lattice $m^{-1} \bZ^d$. 

In $d$ many dimensions, we have the product formula
\begin{equation*}
\label{GenProd}
|\hat \mu(\bxi)| = \prod_{j=1}^\infty g(\bxi/q^j)
\qquad
(\bxi \in \bR^d),
\end{equation*}
where
\[
g(\by) = \left| \sum_{\bj \in \{ 0, 1, \ldots, q-1 \}^d} p_\bj e(\bj \cdot \by)
\right|.
\]
Thus, $\| m \bx \| = 0$ whenever $g(\bx) = 1$. Indeed, this implies that $e(\bj\cdot\bu^{(i)})$ all have the same phase hence
\[
\bx\cdot (\bu^{(i)}-\bu^{(1)})\in\bZ\qquad (2\leq j\leq r)
\]
Thus we see that $\bx\in m^{-1}\bZ^d$ due to the argument above. Now, by continuity and compactness, there exists $c = c(\mu) > 0$ such that
\[
\| m \bx \| \ge 1/q^2
\quad
\Longrightarrow 
\quad
g(\bx) \le e^{-c}.
\]
It follows that $\hat \mu$ is $(q,m)$-good in each variable.
\end{proof}

\subsection{Inputs relating to probability/measure theory}

The following famous results can be found in many places, including \cite{BV2023}.

\begin{lem} [First Borel--Cantelli lemma]
Let $\mu$ be a probability measure, and let $\cE_1, \cE_2, \ldots$ be events such that
\[
\sum_{n=1}^\infty \mu(\cE_n) < \infty.
\]
Then
\[
\mu\left(
\limsup_{n \to \infty} \cE_n \right) = 0.
\]
\end{lem}

Deducing positive measure from the divergence of the series requires, additionally, some form of independence.

\begin{lem} [Divergence Borel--Cantelli]
\label{DBC}
Let $\mu$ be a probability measure, and let $\cE_1, \cE_2, \ldots$ be events such that
\[
\sum_{n=1}^\infty \mu(\cE_n) = \infty
\]
and
\[
\sum_{n,n' \le N} \mu(\cE_n\cap \cE_{n'}) \ll \Bigl( \sum_{n \le N} \mu(\cE_n) \Bigr)^2.
\]
Then
\[
\mu \left(
\limsup_{n \to \infty} \cE_n \right) > 0.
\]
\end{lem}

Then, to upgrade from positive measure to full measure, we leverage the ergodicity of Bernoulli systems.
We eschew the popular tool \cite[Proposition 1]{BDV2006}, as it requires doubling.

As part of our argument, we use the following real-analytic device.

\begin{lem}
\label{lem: divergent reduction} Suppose
\[
\sum_{n=1}^\infty \psi(n) = \infty.
\]
Then there exists a non-increasing function $\tilde \psi: \bN \to [0,1]$ such that
\[
\tilde \psi \le \psi,
\qquad
\tilde \psi = o(\psi),
\qquad
\sum_{n=1}^\infty \tilde \psi(n) = \infty.
\]
\end{lem}

\begin{proof} By induction, there exists a sequence $K_1 = 1, K_2, \ldots$ of positive integers such that
\[
\sum_{K_i \le n < K_{i+1}} \psi(n) \ge i
\]
for all $i$. Now let
\[
\tilde \psi(n) = i^{-1}
\psi(n)
\qquad (K_i \le n < K_{i+1}).
\]
\end{proof}

Whereas the classical abstract zero--full laws of Cassels \cite{Cas1950} and Gallagher \cite{Gal1961} apply to a single approximating function, our variant requires all non-increasing approximating functions at once.

\begin{lem} [Variant of abstract zero--full law]
\label{ZeroFull}
Let $p \ge 2$ and $q \ge 3$ be coprime integers, and let $\mu$ be a missing-digit measure corresponding to a missing-digit set $\cK \subseteq [0,1]^d$ in base $q$. Suppose that \mbox{$\mu(\cW_p(\psi, d)) > 0$} for all non-increasing functions $\psi: \bN \to [0,1]$ with
\[
\sum_{n=1}^\infty \psi(n)^d = \infty.
\]
Then $\mu(\cW_p(\psi, d)) = 1$ for all such $\psi.$
\end{lem}

\begin{proof}
Let $\tilde{\psi}$ be as given by Lemma \ref{lem: divergent reduction}, applied to $\psi^d$ and then taking the $d^{\mathrm{th}}$ root. We see from the hypothesis that $\mu(\cW_p(\tilde{\psi}, d)) > 0.$ For convenience, we write
\[
\cW(\psi) = \cW_p(\psi, d) \mmod \bZ^d,
\qquad
\cW(\tilde \psi) = \cW_p(\tilde \psi, d) \mmod \bZ^d
\]
for the remainder of this proof.

Now let $\ell \geq 0$ be an integer, and consider the set
\[
q^\ell \cW(\tilde \psi) = 
\{ q^\ell \bx: \bx \in \cW(\tilde\psi)\}.
\]
Given any $\bx \in \cW(\tilde\psi),$ there are infinitely many integers $n$ such that $p^n \bx = \bz^{(n)} + \be^{(n)}$ with $\bz^{(n)} \in \bZ^d$ and $|\be^{(n)}|< \tilde\psi(n)$, whence,
\[
p^n q^\ell \bx = q^\ell \bz^{(n)} + q^\ell \be^{(n)}.
\]
Observe that $|q^\ell \be^{(n)}| < \psi(n)$ for any sufficiently large $n$, so
\[
q^\ell \bx \in \cW(\psi).
\]

Consider the set
\[
\tilde\cW := \bigcup_{\ell \geq 0} q^\ell \cW(\tilde \psi) \subseteq \cW(\psi).
\]
It is clear that $\tilde \cW$ is Borel, and so $\mu(\tilde \cW) \ge \mu(\cW(\tilde \psi)) > 0$.
Moreover, the symmetric difference satisfies
\begin{equation}
\label{inv}
\mu(\tilde \cW \Delta (\times q)^{-1} \tilde \cW) = 0.
\end{equation}
Indeed, we plainly have $\tilde \cW \subseteq (\times q)^{-1} \tilde \cW$. Since $\mu$ is $\times q$-invariant, we see that
\[
\mu(\tilde \cW)=\mu((\times q)^{-1}\tilde \cW),
\]
so we must have \eqref{inv}.

Observe that $(\cK,\times q,\mu)$ can be considered as a Bernoulli scheme via the left shift in the space of one-sided infinite digit sequences. In particular it is ergodic, so $\mu(\tilde\cW)=1$. Thus, we also have 
\[
\mu(\cW_p(\psi, d)) = \mu(\cW(\psi)) = 1.
\]
\end{proof}

\section{Main technical estimate}
\label{tec}

In this section, we establish the following generalisation of Theorem~\ref{AvgDecay}. Note that here the Fourier transform is used. In the case $\ell = 0$, this coincides with the Fourier coefficient, see \eqref{coincide}.

\begin{thm}
\label{AvgDecay2}
Let $p \ge 2$ be an integer, and let $q \ge 3$ be a prime such that $q \nmid p$. Let $f: \bR^d \to [0,1]$ be $(q,m)$-good in each variable, for some $m \in \bN$, and fix $c > 0$ sufficiently small in terms of $p,q,f$. Let $Q \ge 1$, and let $\ell \ge 0$ be an integer, and let $\bt \in \bZ^d$ with
$
1 \le |\bt| \le Q^2$.
Then, for any fixed $\tau \in (0,50)$,
\[
\sum_{\substack{Q < k \le 2Q \\
\log \frac{mp^k|\bt|}{q^\ell}
> Q^{(c/\log\log Q)^{\tau}} }} f(p^k \bt / q^\ell ) \ll_{p,q,m,f,\tau}
Q^{-(c/\log\log Q)^{1+\tau}}.
\]
\end{thm}
Let us now prove this. We may assume that $Q$ is arbitrarily large. 
Let $K = \lfloor k_0 \rfloor$, where $mp^{k_0}|\bt| = q^\ell$. Writing $k - K = h \in \bN$ it suffices, by dyadic pigeonholing, to prove that if
$Q^{(c/ \log \log Q)^\tau} \le H \le Q$ then
\begin{align}\label{eqn: thm31key}
\sum_{H < h \le 2H} f(p^{K+h} \bt/q^\ell) \ll H^{-2c/\log \log Q}.
\end{align}
Indeed, summing over these dyadic values of $H$ gives
\[
\sum_{H}H^{-2c/\log\log Q}\ll (\log\log Q)Q^{-2(c/\log\log Q)^{1+\tau}}\ll Q^{-(c/\log\log)^{1+\tau}}.
\]

We choose parameters
\[
A = (\log \log Q)^{100},
\qquad
L = \frac{H}{(\log Q)^{2A}},
\qquad
\eta = (\log \log Q)^{-50}.
\]
With $C$ a large, positive constant and $\tet = C \log Q$, we consider Stewart intervals
\[
\Tet_j = [\tet^{j-1}, \tet^j],
\]
where 
\[
2 \mid j,
\qquad
\tet^{j-1} \ge \frac{H}{(\log Q)^A},
\qquad
\tet^j \le \frac{H \log p}
{\log \max 
\{ p,q \}}.
\]
There are at most $A$ many and at least $\Ome_{p,q,C}(A)$ many such intervals,
which we call \emph{chosen intervals}.

A vector $\bn \in \bZ^d$ has a \emph{generalised digit change} in base $q$ in position $j$ if some $|n_i|$ has a generalised digit change in base $q$ in position $j$. Let $j_1, \ldots, j_R$ be the positions of the generalised digit changes of $mp^{K+h} \bt$ in base $q$. By Theorem \ref{Stewart},
each chosen interval $\Tet_j$ intersects
$\{ j_1, \ldots, j_R \}$, since the only digit changes in base $p$ occur in the first $O_m(\log Q)$ many positions.

\subsection{Qualification stage}

A generalised digit change $j_i$ is \emph{qualified} if there are no generalised digit changes in positions 
\begin{equation}
\label{positions}
j_i + 1, \ldots, 
j_i + \left \lfloor
\frac{H}{(\log Q)^{A}}
\right \rfloor.
\end{equation}
The generalised digit change, together with the maximal block of constant digits is called an \emph{empty carriage}. An empty carriage can never be placed through an entire chosen interval, because this would force that interval to be free of generalised digit changes.

A chosen interval $\Tet_j$ is \emph{qualified} if it intersects at least one empty carriage. Let $c_1$ be a positive constant, small depending on $p,q,C$. We say that $h$ is \emph{qualified} if the base $q$ expansion of $mp^{K+h} \bt$ contains at least $c_1 A$ many qualified intervals.

Suppose $h \in (H,2H]$ is unqualified, where 
\[
Q^{(c/\log \log Q)^\tau} \le H \le Q.
\]
Then there are  $\Ome(A)$ many unqualified intervals. Each unqualified interval contains a digit change in base $q$ within 
any subinterval of length $\lfloor H/(\log Q)^A \rfloor$, and therefore at least
\[
\left \lfloor
\frac{\tet^j - \tet^{j-1}}
{H/(\log Q)^A}
\right \rfloor \ge \tet - 2
\]
many digit changes in base $q$ in total. Moreover, the condition $2 \mid j$ on the chosen intervals ensures that the digit changes obtained in this way are pairwise distinct. Thus, 
there are $\Ome(A \log Q)$ many digit changes in total that are not in the final $\ell$ many positions. As $f$ is $(q,m)$-good in each variable,
\[
f(p^{K+h} \bt / q^\ell) \ll Q^{-\Ome(A)},
\]
whence,
\[
\sum_{h \text{ unqualified}} f(p^{K+h} \bt / q^\ell) \ll H Q^{-\Ome(A)} \ll Q^{1-(\log\log Q)^{99}} \ll Q^{-2c/\log \log Q}.
\]

\subsection{Exceptional stage}
To deal with qualified $h$, we partition the chosen intervals into subintervals of length at most $L$. To each qualified $h\in (H,2H]$, we associate the interval $\cI_h =[h,h+L-1] \cap (H,2H].$ The union of all such intervals covers the set of qualified elements in $(H,2H]$. We may greedily choose a disjoint collection of such intervals covering the set of qualified $h \in (H,2H]$. There are at most $(\log Q)^{2A} + 1$ many such intervals parametrised, say, by $\cH$. 

Let us now fix one such interval $\cI_h$.
We first observe that the qualified chosen intervals
are associated with $\Ome(A)$ many distinct empty carriages,
where each empty carriage is taken to be maximal.
Indeed, an empty carriage can intersect at most
two chosen intervals, since otherwise it would
contain a  Stewart interval with no
digit changes in base $q$, contradicting
Theorem~\ref{Stewart}. For each empty carriage, choose
a co-ordinate in which its initial generalised
digit change occurs. The subsequent digits in
this co-ordinate are all equal to some
\mbox{$v\in\{0,q-1\}.$} By pigeonholing over the
co-ordinate and $v$, we retain $N \gg A$ many empty
carriages of the same type. For simplicity, we assume that $v = 0$ here. The case $v = q-1$ is similar via Lemma \ref{Act3}. Their starting
positions are separated by at least
$\lfloor H/(\log Q)^A\rfloor$.

Let $T \in \bN$ be minimal such that
$L \le H_T(p)$. By Lemma \ref{Occ2},
\[
q^T \asymp_q L.
\]
Since $T \asymp \log L = o(H/(\log Q)^A)$,
the length-$T$ boundary strings immediately
preceding these empty carriages are pairwise disjoint. Later, when multiplying by $p$, these empty carriages will start to have digit changes at the right end.

Put 
\[
X_0 = |mp^{K+h}t_j|.
\]
By Lemma \ref{Act2}, if $\Del \in \{ 0, 1, \ldots, L-1 \}$ then the base $q$ expansion of $p^\Del X_0$ has the form
\begin{align*}
* \langle p^\Del y_1 \mmod q^T \rangle 0 \ldots 0 *
\langle p^\Del y_2 \mmod q^T \rangle 0 \ldots 0 \langle p^\Del z \rangle,
\end{align*}
where each string of zeros is non-trivial, and $z$ has at least $\ell$ many digits.
Moreover, each $\langle p^\Del y_i \mmod q^T \rangle$ realises exactly $L$ many possibilities as $\Del$ varies.

Put $\del = 2 \eta (1 + \log (6/\eta))$. Recall that to each qualified digit change we associated $X_0$ such that the digits in positions \eqref{positions} are all zero. To each retained empty carriage, we associate a \emph{boundary string} consisting of the $T$ many preceding digits in the base-$q$
expansion of $p^\Del X_0$. We call $\Del$ \emph{exceptional} if, for at least one
of the $N$ many retained empty carriages, the associated
boundary string has fewer than
$\eta T$ many generalised digit changes.

If $\Del$ is unexceptional, then there are $N \gg A$ many empty carriages having more than $\eta T$ many digit changes, so
\[
f(p^{K+h+\Del} \bt / q^\ell) \ll e^{-\Ome(N \eta T)} = e^{-\Ome((\log \log Q)^{50-\tau}\log Q)}
\ll Q^{-101}.
\]
Thus, the unexceptional contribution to
\[
S_h := \sum_{\Delta=0}^{L-1} f(p^{K+h+\Delta}\bt/q^\ell)
\]
is $O(Q^{-100})$.

By Lemma~\ref{EK} and the union bound, the number of exceptional $\Del$ is bounded above by
\begin{align*}
Nq^{\del T} &\le (\log \log Q)^{100} O(H/(\log Q)^{2A})^{2\eta(1+\log(6/\eta))}  \\
&\ll H^{O(\log \log \log Q)/(\log \log Q)^{50}}=H^{o(1/\log\log Q)}.
\end{align*}
By Stewart's theorem (Theorem \ref{Stewart}),
\[
f(p^{K+h+\Del} \bt / q^\ell) \ll e^{-4c \log H/ \log \log Q} = H^{-4c/\log \log Q}.
\]
Thus, the exceptional contribution to $S_h$ is $O(H^{-3c / \log \log Q})$. 

Recall that there are at most $(\log Q)^{2A} + 1$ many intervals $\cI_h$ comprising our disjoint union. Since
\begin{align*}
(\log Q)^{2A} &= H^{2(\log\log Q)^{101}/\log H} \leq H^{2(\log\log Q)^{101}/((c/\log\log Q)^\tau)\log Q)}\\
&\ll H^{c/\log\log Q},
\end{align*}
we conclude that
\begin{align*}
\sum_{h \text{ qualified}} f(p^{K+h}/q^\ell) \leq \sum_{h \in \cH} S_h \ll (\log Q)^{2A} H^{-3c/\log\log Q} \ll H^{-2c/\log\log Q}.
\end{align*}

This establishes \eqref{eqn: thm31key} and completes the proof of Theorem \ref{AvgDecay2}.

\subsection{A simple  extension}

We need the following extended result which is a consequence of the qualification argument. The result will be used in \S \S \ref{sec: div}--\ref{sec: QN}. Notice that $\bu$ below can be much larger than $Q^{O(1)}$.

\begin{lem}[Uniform two-block Fourier estimate]
\label{lem: two_blocks}
Let $p\ge 2$, let $q\ge 3$ be prime with $q\nmid p$, and
let $\mu$ be an affinely irreducible missing-digit measure in
base $q$ on $[0,1]^d$. Fix $B\ge 1$. Then there is $c_B>0$
such that, uniformly for $Q$ sufficiently large,
$\ba,\bb\in\bZ^d$ and $s\in\bZ$ satisfying
\[
\max\{|\ba|,|\bb|\}\le Q^B,\qquad
0\le s\le Q^B,\qquad \bu:=\ba+p^s\bb\ne\bzero,
\]
we have
\[
\sum_{Q<k\le 2Q}\bigl|\hat\mu(p^k\bu)\bigr|
\ll_{B,p,\mu} Q^{-c_B/\log\log Q}.
\]
\end{lem}

\begin{proof}
Choose a co-ordinate $i$ with $|u_i|=|\bu|$, and put
$n=m|u_i|$, where $m$ is as in Lemma~\ref{decent}.
We first record the base-$p$ structure of $n$.
Put 
\[
w = \lceil \log_p(2mQ^B +2) \rceil \ll \log Q.
\]
If $s\le 2w$, or if $b_i=0$, then $n$ has
$O(\log Q)$ many base-$p$ digits. Otherwise, after
changing sign if necessary, we may write
\[
n=b_0p^s+a_0,\qquad 1 \le b_0 \ll Q^B, \quad
|a_0| < p^w.
\]
If $a_0\ge 0$, its expansion consists of a short
leading block, zeros, and a short final block. If
$a_0<0$, use
\[
n=(b_0-1)p^s+(p^{s-w}-1)p^w+(p^w-|a_0|).
\]
Here the middle block consists entirely of digits
$p-1$. In either case, the base-$p$ digit
changes of $p^kn$ lie in at most two ranges of
$O(\log Q)$ many consecutive positions, independently
of $k$ and $s$. We call these \emph{$p$-mysterious ranges}.

The number $r_k$ of base-$p$ digits of $p^kn$
is given by $r_k=k+\lfloor\log_p n\rfloor+1$.
Choose
\[
\theta=C\log\bigl(Q+\lfloor\log_p n\rfloor+2\bigr)
\asymp_B\log Q,
\]
where $C$ is as in Theorem~\ref{Stewart}. Since
$Q<k\le 2Q$, this choice satisfies the condition
on $\theta$ in that theorem. Take its even-indexed
Stewart intervals $\Theta_j = [\tet^{j-1}, \tet_j]$ satisfying
\[
\theta^{j-1}\ge R:=\left\lfloor
\frac{Q}{(\log Q)^A}\right\rfloor,
\qquad
\theta^j\le\frac{Q\log p}{\log\max\{p,q\}},
\qquad A=(\log\log Q)^{100},
\]
of which there are $\Ome(A)$ many, and call these \emph{chosen intervals}. The first $p$-mysterious range precedes
them, and the second meets at most $O(1)$ many
Stewart intervals.  Discard the chosen intervals
meeting such a $p$-mysterious range, together with their
neighbours across an intervening odd-indexed
interval meeting a $p$-mysterious range. Only $O(1)$
many chosen intervals are lost. 

By
Theorem~\ref{Stewart}, every remaining chosen
interval contains a base-$q$ digit change; moreover,
between any two successive remaining chosen intervals
there is a Stewart interval containing
a base-$q$ digit change. Applying the same observation to all
Stewart intervals between $O(\log Q)$ and a
fixed multiple of $Q$ gives, uniformly in $k$,
$$\Ome(\log Q/\log\log Q)$$
many base-$q$ digit changes.
In particular, Lemma~\ref{decent} furnishes
\begin{equation}
\label{eq: two_blocks_pointwise}
\hat\mu(p^k\bu) \ll Q^{-c_0/\log\log Q}
\qquad(Q<k\le2Q),
\end{equation}
for some constant $c_0>0$.

We now repeat the qualification argument from the
proof of Theorem~\ref{AvgDecay2}, using only the
chosen intervals not meeting a $p$-mysterious range.
Set
\[
L=\left\lfloor\frac{Q}{(\log Q)^{2A}}\right\rfloor,
\qquad \eta=(\log\log Q)^{-50}.
\]
If $k$ is unqualified, then $\Ome(A)$ many chosen intervals
contain $\Ome(\theta)$ many generalised digit changes each so, by Lemma~\ref{decent},
their total contribution to the 
sum at hand is negligible.

Partition $(Q,2Q]$ into intervals of length at most
$L$, and fix the first qualified integer $k_0$ in
one such interval. We consider only shifts $v$ for
which $k_0+v$ remains in that interval. As in
\S 3.2, it has
$N\gg A$ many distinct maximal constant runs of length
at least $R$, all of the same type $0$ or $q-1$.
Let $U$ be minimal with $H_U(p)\ge L$. By the
prime-base occupancy lemma (Lemma \ref{lma: local occ}), we have $q^U\asymp L$ and
$U\asymp\log Q$. Since $R\gg L$, Lemmas~\ref{Act2}
and~\ref{Act3} apply to these runs for every
$0\le v<L$. Their associated length-$U$ boundary
strings are disjoint, and each takes distinct
values as $v$ varies.

Put $\varepsilon=2\eta(1+\log(6/\eta))$.
Lemma~\ref{EK} and the union bound show that at
most $Nq^{\varepsilon U}=Q^{o(1/\log\log Q)}$
values of $v$ have fewer than $\eta U$ many generalised
digit changes in some selected boundary string.
For every other $v$, the $N$ disjoint strings
supply $\Ome(N\eta U) \gg (\log\log Q)^{50}\log Q$ many
generalised digit changes, and hence
\[
\hat\mu(p^{k_0+v}\bu) \ll Q^{-100}.
\]
For the exceptional $v$, we use
\eqref{eq: two_blocks_pointwise}. Thus, the
contribution of each interval of length $L$ is
$O(Q^{-c_0/(2\log\log Q)})$. Finally,
\[
\frac{Q}{L}\ll (\log Q)^{2A}
=Q^{o(1/\log\log Q)},
\]
so summing over the intervals proves the lemma.
\end{proof}

\section{Convergence theory}
\label{sec: ETP}

In this section, we use Theorem \ref{AvgDecay2} to prove the convergent part of Theorem~\ref{MainThm}. Observe that
\[
\cW_p(\psi, d) \cap [0,1]^d = \limsup_{k \in \bN} \cA_k,
\]
where
\[
\cA_k = \{ x \in [0,1]^d: \| p^k x \| < \psi(k) \}.
\]

\subsection{First moment estimate}

Let $Q \in \bN$, and let $\delta_Q > 0$ with
\begin{equation}
\label{delLower}
\delta_Q^{-d} \leq Q \log^2 Q.  
\end{equation}
The key is to bound
\[
\sum_{k=Q+1}^{2Q} \mu(\cA_k).
\]
To achieve this, we consider
\[
N_Q := \sum_{k=Q+1}^{2Q} \int A_k(x) \d \mu(x),
\]
where $A_k$ is a suitable bump function majorising or minorising the indicator function of $\cA_k$.
More precisely, recall the construction of $\ome_\delta$ from \S \ref{prep}. We choose $\delta = \delta_Q/p^k$ and define, for $\bx \in \bR^d$,
\[
A_k(\bx) = \sum_{\ba \in \bZ^d} \ome_{\delta_Q/p^k} \left(\bx -\frac{\ba}{p^k} \right).
\]

By Parseval, we have the Fourier series expansion
\[
\int A_k(x) \d \mu(x) = \sum_{\bxi \in \bZ^d} \hat{A}_k(\bxi)\hat{\mu}(-\bxi).
\]
This sum converges absolutely, owing to the rapid decay of $\hat{A}_k.$ Observe that if $\hat{A}_k(\bxi) \neq 0$ then $\bxi$ is a multiple of $p^k.$ Hence, we may write the sum above as
\[
\sum_{\bt\in\bZ^d}
\hat{A}_k(p^k \bt) \hat{\mu}(-p^k\bt).
\]
We will use the facts that for any $\bt,$
\[
\hat{A}_k(p^k\bt) \ll \delta^d_Q
\]
and, for $|\bt|\geq 1/\delta_Q$ and any $N>0,$
\[
\hat{A}_k(p^k\bt) \ll (\delta_Q |\bt|)^{-N} \delta_Q^d.
\]
We split the sum $N_Q$ as follows:
\[
N_Q = \sum_{k=Q+1}^{2Q} \left(\sum_{|\bt|\leq 1/\delta_Q} \hat{A}_k(p^k \bt) \hat{\mu}(-p^k\bt) + \sum_{|\bt|>1/\delta_Q} \hat{A}_k(p^k \bt) \hat{\mu}(-p^k\bt) \right).
\]
We address the two parts separately and show that, subject to \eqref{delLower}, we have
\begin{align}
\label{eqn: ETP bound}
N_Q = \delta^d_Q (c_\ome Q + o(Q))
\end{align}
for some constant $c_\ome > 0.$

\bigskip
\textbf{Case: $|\bt|\leq 1/\delta_Q$.} Observe that for each $k$, there is a contribution of $c_\ome Q \delta^d_Q$ coming from the $\bt = \bzero$ term, 
where
\[
c_\ome = \| \ome \|_1 > 0
\]
is constant.
Therefore
\[
\sum_{k=Q+1}^{2Q}\sum_{|\bt|\leq 1/\delta_Q} \hat{A}_k(p^k \bt) \hat{\mu}(-p^k\bt) = \delta^d_Q \left(c_\ome Q + O(E_Q) \right),
\]
where
\[
E_Q =  \sum_{0<|\bt|\leq 1/\delta_Q} \
\sum_{k=Q+1}^{2Q} 
|\hat{\mu}(p^k \bt)|.
\]
Using \eqref{delLower}, we may apply Theorem \ref{AvgDecay2} with $\ell = 0$ and $\tau=1$, together with Lemma \ref{decent} to see that for some constant $c>0$,
\begin{align}
\label{EQ}
E_Q &\ll \del_Q^{-d} Q^{-(c/\log\log Q)^2} \ll Q (\log Q)^2 Q^{-(c/\log\log Q)^2}
= o(Q).
\end{align}

\bigskip
\textbf{Case: $|\bt| > 1/\delta_Q$.} 
Given the rapid decay of $\hat{A}_k$, the large Fourier coefficients do not contribute significantly to $N_Q$. This is an instance of the uncertainty principle. To make the argument precise, we split the sum over $\bt$ further:
\[
\sum_{u \geq 1} \
\sum_{2^u \geq |\bt| \delta_Q > 2^{u - 1 }} \
\sum_{k=Q+1}^{2Q}
\hat{A}_k(p^k \bt) \hat{\mu}(-p^k\bt).
\]
By Theorem \ref{AvgDecay2}, if $2^u/\del_Q \le Q^2$ then
\begin{align*}
\sum_{2^u \geq |\bt| \delta_Q > 2^{u - 1 }} \
\sum_{k=Q+1}^{2Q}
\hat{A}_k(p^k \bt) \hat{\mu}(-p^k\bt)
&\ll \left(\frac{2^u}{\delta_Q} \right)^d 2^{-N u} \delta^d_Q  Q^{-(c/\log\log Q)^2} \\
&= (2^{d-N})^u Q^{-(c / \log \log Q)^2}.
\end{align*}
Choosing $N = d + 2$, and recalling \eqref{delLower}, we find that
\[
\sum_{2^u/\del_Q \le Q^2} \
\sum_{2^u \geq |\bt| \delta_Q > 2^{u - 1 }} \
\sum_{k=Q+1}^{2Q}
\hat{A}_k(p^k \bt) \hat{\mu}(-p^k\bt) = o(Q\del_Q^d).
\]

For those $u$ with $2^u/\delta_Q \geq Q^2$, we note that
\begin{align*}
\sum_{2^u \geq |\bt| \delta_Q >2^{u-1}} \
\sum_{k=Q+1}^{2Q} \hat{A}_k(p^k \bt) \hat{\mu}(-p^k\bt) &\ll 2^{-Nu} \delta^d_Q\sum_{k=Q+1}^{2Q} \
\sum_{|\bt|\leq 2^u / \delta_Q}1 \\ &\ll_d Q 2^{-N u + d u}.
\end{align*}
Choosing $N = d + 2$ ensures, by \eqref{delLower}, that 
\[
Q \sum_{2^u > Q^2\delta_Q}
2^{-N u + d u} = o(Q \del_Q^d).
\]
Collecting the estimates above yields
\[
\sum_{u \geq 1} \
\sum_{2^u \geq |\bt| \delta_Q > 2^{u -1}} \ \sum_{k=Q}^{2Q}  \hat{A}_k(p^k \bt) \hat{\mu}(-p^k\bt) =
o(Q\delta_Q^d).
\]
\bigskip

We have handled both parts of $N_Q$, confirming \eqref{eqn: ETP bound}.

\subsection{Proof of Theorem \ref{MainThm},  convergent part}

Given $\psi$ monotonic with
\[
\sum_{k=1}^\infty \psi(k)^d < \infty,
\]
we may inflate $\psi$ on dyadic intervals. More precisely, let $\ell \in \bN$ and consider $2^\ell < k \le 2^{\ell+1}.$ Denote $Q=2^\ell$ and define
\[
\tilde{\psi}(k) = \max \left \{ \psi(k), \frac{1}{Q^{1/d} \log^{2/d} Q} \right\}.
\]
For $k \in (Q,2Q]$,
\[
\frac{1}{Q^{1/d} \log^{2/d} Q}\leq \tilde \psi(k) \leq \max \left \{ \psi(Q),\frac{1}{Q^{1/d} \log^{2/d} Q} \right\} =: \delta_Q.
\]
Crucially,
\[
\sum_{\ell \geq 1} 2^\ell\delta^d_{2^\ell} \ll \sum_{k\geq 1} \psi(k)^d + \sum_{k\geq 1} k^{-2} < \infty. 
\]
Now by \eqref{eqn: ETP bound},
\[
\sum_{k\geq 3} \mu(\cA_k) \leq \sum_{k \geq 3} \mu(\tilde\cA_k) = \sum_{\ell \geq 1} \sum_{k=2^\ell + 1}^{2^{\ell+1}} \mu(\tilde\cA_k) \ll_d \sum_{\ell \geq 1} 2^{\ell} \delta_{2^\ell}^d < \infty.
\]
An application of the first Borel--Cantelli lemma completes the proof.

\section{Divergence theory}\label{sec: div}

In this section, we establish the divergent part of Theorem \ref{MainThm}. We will show, assuming
\begin{equation}
\label{Div1}
\sum_{k=1}^\infty \psi(k)^d = \infty,
\end{equation}
that $\mu(\cW_p(\psi,d))>0.$ The desired conclusion will then follow from Lemma \ref{ZeroFull}.

For convenience, we will assume that each $\cA_k$ is a union of $p$-boxes. A $p$-box is an abbreviation for a $(p,k,h)$-box, which we define as an axis-aligned rectangle in $\bR^d$ with the following properties:
\begin{itemize}
\item
The centres are elements of $p^{-k} \bZ^d$. 
\item 
The side lengths are $1/p^{k+h}$.
\end{itemize}
Here $h \in \bN$ will have the additional property --- not used to define $p$-boxes --- that
\[
p^{-h}/2 = \psi(k) \ge (k\log^2 k)^{-1/d}.
\]
We will also take $\psi$ to be constant on dyadic ranges so that, in the analysis, the parameter $h$ will not depend on $k$.
Later, we will see why these assumptions are permissible.

\bigskip

By \eqref{eqn: ETP bound} with $\del_Q = \psi(Q)$,
\[
\sum_{Q < k \le 2Q} \mu(\cA_k) \asymp Q \psi(Q)^d.
\]
Using \eqref{Div1}, we may now conclude that
\begin{equation}
\label{DivReq}
\sum_{k \ge 1} \mu(\cA_k) = \infty.
\end{equation}

\bigskip

Our task is to establish quasi-independence on average, specifically,
\begin{equation}
\label{QIA}
\sum_{k, k' \le T} \mu(\cA_{k} \cap \cA_{k'}) \ll \sum_{k,k'\le T} \mu(\cA_{k}) \mu(\cA_{k'}).
\end{equation}
Our approach is motivated by the \emph{moment transference principles} framework from \cite{CY1}. The present setting is more challenging as we require two different methods to 
estimate the mean overlap: one based in physical space, and one based in frequency space.
We will  compare $\sum_{k,k'}\mu(\cA_k\cap \cA_{k'})$ to $\sum_{k,k'} \lam(\cA_k\cap \cA_{k'})$, showing that they are asymptotically close under mild conditions. Invoking results from \cite{PVZZ2022,TZ} to the measure $\lam$ will then deliver the claimed estimate for $\mu$. (Those results do not apply directly to $\mu$ because $\mu$ lacks pointwise Fourier decay.)

\subsection{Fourier-analytic method} 

As in \S \ref{sec: ETP}, let $A_k$ be a suitable bump function majorising the indicator function of $\cA_k$. By expanding $A_k$ and $A_{k'}$ into Fourier series, we find that
\[
\int A_k(x) A_{k'}(x) \d \mu(x) =  \sum_{\bt, \bt' \in \bZ^d} \hat A_k(p^k \bt) \hat A_{k'} (p^{k'} \bt') \hat \mu(-p^k\bt - p^{k'} \bt')
\]
and
\begin{align*}
\int A_k(x) A_{k'}(x) \d \lam(x) &= \sum_{\bt, \bt' \in \bZ^d} \hat A_k(p^k \bt) \hat A_{k'} (p^{k'} \bt') \hat \lam(-p^k\bt - p^{k'} \bt') \\
&= \sum_{p^k \bt + p^{k'} \bt' = \bzero} \hat A_k(p^k \bt) \hat A_{k'} (p^{k'} \bt').
\end{align*}
In particular,
\begin{align*}
V(k,k') &:= \int A_k(x) A_{k'}(x) \d \mu(x) -
\int A_k(x) A_{k'}(x) \d \lam(x) \\ &=
\sum_{p^k \bt + p^{k'} \bt' \ne \bzero} \hat A_k(p^k \bt) \hat A_{k'} (p^{k'} \bt') \hat \mu(-p^k\bt - p^{k'} \bt').
\end{align*}

It is convenient to introduce the bookkeeping factor $2^{-{u_\bt}N}$, where 
$u_\bt = 0$ if
\mbox{$|\bt| \le p^h$,} and otherwise $u_\bt$ is the unique positive integer satisfying 
\[
2^{u_\bt - 1} < p^{-h} |\bt| \le 2^{u_\bt}.
\]
Here $N$ is an arbitrary constant, and we define $u_{\bt'}$ analogously.
We obtain
\begin{equation}
\label{eqn: IntersectionFourier}
\frac{V(k,k')}
{\lam(\cA_k) \lam(\cA_{k'})} \ll \sum_{p^k \bt + p^{k'} \bt' \ne \bzero} 2^{-N(u_{\bt}+u_{\bt'})} |\hat{\mu}(p^k \bt+p^{k'}\bt')|.
\end{equation}
This bound is useful when $|k-k'|$ is not too large. We will consider the sum over $k$ and $k'$. When $k$ and $k'$ are close, this double-sum is essentially a single sum, to which we can apply Theorem \ref{AvgDecay2}. Let us now elaborate.

Let $Q \ge 100$ and
\[
\Delta = \frac{Q^{(c/\log\log Q)^2/2}}{\log Q}.
\]
We claim that
\begin{align}
\notag
&\sum_{Q < k \le 2Q} \
\sum_{k \le k'\leq k + \Delta} \ \sum_{p^k \bt + p^{k'} \bt' \ne \bzero} 2^{-N(u_{\bt}+u_{\bt'})}
|\hat{\mu}(p^k\bt+p^{k'} \bt')| \\
\label{claim: Delta bounded sum}
&
\ll Q^{-(c/\log\log Q)^2/2}p^{d(h+h')}.
\end{align}

\begin{proof}[Proof of claim]
The left hand side equals
\[
\sum_{0 \le \delta\leq \Delta} \ 
\sum_{\bt + p^{\delta} \bt' \ne \bzero} 2^{-N(u_{\bt} + u_{\bt'})} \sum_{Q < k \le 2Q} |\hat{\mu}(p^k(\bt+p^{\delta}\bt'))|.
\]
We split the sum  according to whether $\max 
\{ |\bt|, |\bt'| \} \geq Q^{3}$. 

First, suppose $\max \{ |\bt|,|\bt'|\}\leq Q^{3}$. We use Lemma \ref{lem: two_blocks}. Consequently,
\begin{align*}
&\sum_{0\leq \delta \leq \Delta} \ \sum_{\substack{
\max \{ |\bt|, |\bt'| \} \leq Q^{3} \\ \bt+p^{\delta} \bt' \ne \bzero
}}
2^{-N(u_{\bt}+u_{\bt'})} \ \sum_{Q < k \le 2Q} 
|\hat{\mu}(p^k(\bt + p^{\delta} \bt'))| \\
&\ll Q^{-(c/\log\log Q)^2/2} \sum_{u,u'\geq 0} 2^{-N(u+u')} \sum_{\substack{
\bt,\bt': \\(u_\bt,u_{\bt'}) =(u,u')}} 1 \\
&\ll
Q^{-(c/\log\log Q)^2/2} \sum_{u,u'\geq 0} 2^{(d-N)(u+u')}
p^{d(h+h')}.
\end{align*}

In the case $\max \{ |\bt|, |\bt'|\} \geq Q^{3}$, the trivial bound $\|\hat{\mu}\|_\infty \leq 1$ yields
\begin{align*}
&\sum_{Q < k, k' \le 2Q} \ \sum_{\max \{ |\bt|, |\bt'|\} \geq Q^{3}} 2^{-N(u_{\bt} + u_{\bt'})}|
\hat{\mu}(p^k\bt+p^{k'} \bt')| \\
&\leq Q^2 \sum_{\substack{u,u'\geq 0 \\
\max \{ p^h 2^{u}, p^{h'} 2^{u'} \} \ge Q^{3}}} \ \sum_{\substack{\bt,\bt': \\ (u_\bt, u_{\bt'}) = (u,u')}} 2^{-N(u+u')} \\
&\ll Q^2 \sum_{\substack{u,u'\geq 0 \\
\max \{ p^h 2^{u}, p^{h'} 2^{u'} \} \ge Q^{3}}} 2^{(d-N)(u+u')} p^{d(h+h')}.
\end{align*}
Choosing $N = 2(d+1)$ confirms the claim.
\end{proof}

By \eqref{eqn: IntersectionFourier} and
\eqref{claim: Delta bounded sum},
\[
\sum_{\substack{Q < k, k' \le 2Q \\ |k-k'| \le \Del}} \frac{V(k,k')}{\lam(\cA_k) \lam(\cA_{k'})} \ll Q^{-(c/\log \log Q)^2/2} p^{d(h+h')} \ll Q^2.
\]
In particular,
\begin{equation}
\label{CloseBound}
\sum_{\substack{Q < k, k' \le 2Q \\ |k-k'| \le \Del}} \mu(\cA_k\cap \cA_{k'}) \ll  \sum_{Q < k,k' \le 2Q} \lam(2\cA_k \cap 2\cA_{k'}) 
+ Q^2 \lam(\cA_{Q+1})^2.
\end{equation}
\subsection{Localisation method} 
We now come to the case in which $k'-k$ exceeds $\Del$. Let $\cI$ be a $p$-box composing $\cA_k$ and intersecting the support of $\mu$, so that $\cI$ contributes to $\mu(\cA_k).$ 
For convenience, we perform a rescaling to make $\cI$ the unit box centred at the origin. This leads to a rescaling of $\mu$.  
We begin with a short description of this rescaling, before going into more detail. Let $\ell \in \bN$ be such that 
\[
1/q \le \gam := q^{-\ell} p^{k+h} < 1.
\]
Before rescaling, the support of $\mu$ is contained in a union of $|\cD|^\ell$ many hypercubes of side length $q^{-\ell}$. This decomposes $\cI$ into branches. Each branch has scale $q^{-\ell}$ before rescaling, and scale $\gamma \asymp 1$ afterwards. Consequently, there are at most $O_\mu(1)$ many rescaled branches of $\mu$ in the unit hypercube.

Concretely, let
\[
\cI = \left[ \frac{a_1 - p^{-h}/2}{p^k}, \frac{a_1 + p^{-h}/2}{p^k} \right] \times \cdots \times
\left[ \frac{a_d - p^{-h}/2} {p^k}, \frac{a_d + p^{-h}/2}{p^k} \right]
\]
and
$
S(\bx) = p^{k+h} \bx - p^h \ba + (1/2, \ldots, 1/2),
$
so that $S(\cI) = [0,1]^d$.
The branches of $\cI$ have the form $\cI \cap \cJ$ with
\[
\cJ = \left[ \frac{b_1}{q^\ell}, \frac{b_1 + 1}{q^\ell} \right] \times \cdots \times
\left[ \frac{b_d}{q^\ell}, \frac{b_d + 1}{q^\ell} \right],
\]
where the digits of $\bb$ in base $q$ lie in $\cD$. 
There are $O_\mu(1)$ many branches $\cJ$ to consider, namely those for which 
$
\mu(\cI \cap \cJ) > 0.
$
Let us fix one of them and, for Borel $\cA$, define
\[
\mu_\cJ(\cA) = 
\frac{\mu(\cA \cap \cJ)}{\mu(\cJ)}
\]
and
\[
\nu(\cA) = 
\mu_\cJ(S^{-1}\cA) =
\frac{\mu(S^{-1} \cA \cap \cJ)}{\mu(\cJ)}.
\]

Next, consider
\[
\cA_{k'} = \bigsqcup_{\ba'}
\left[ \frac{a'_1 - p^{-h'}/2}{p^{k'}}, \frac{a'_1 + p^{-h'}/2}{p^{k'}} \right] \times \cdots \times
\left[ \frac{a'_d - p^{-h'}/2}{p^{k'}}, \frac{a'_d + p^{-h'}/2}{p^{k'}} \right].
\]
We wish to bound
\[
\mu_\cJ(\cA_{k'} \cap \cI)
= \nu(S(\cA_{k'}) \cap [0,1]^d) \leq \nu(S(\cA_{k'})).
\]
Note that
\begin{align*}
S(\cA_{k'}) &=
\bigsqcup_{\bc}
\left[ \frac{c_1 - p^{-h'}/2}{p^{k'-k-h}}, \frac{c_1 + p^{-h'}/2}{p^{k'-k-h}} \right] \times \cdots \times
\left[ \frac{c_d - p^{-h'}/2}{p^{k'-k-h}}, \frac{c_d + p^{-h'}/2}{p^{k'-k-h}} \right] \\ &\qquad + (1/2, \ldots, 1/2)
\end{align*}
is a translation of a disjoint union of $p$-boxes. 
We will see, using the self-similarity of $\mu$, that $\nu$ is a rescaled and translated copy of $\mu$. Our analysis of $\nu(S(\cA_{k'}))$ will proceed along similar lines to \S \ref{sec: ETP}.

To enact the Fourier analysis, we replace the indicator function of $S(\cA_{k'}) \cap [0,1]^d$
by a smooth majorant, which we denote by $f$. This is $p^{k+h-k'}$-periodic and for any fixed $N>0,$
\[
\hat f(\bxi) \ll \lam(S(\cA_{k'}) \cap [0,1]^d) (1 + p^{k+h-k'-h'} |\bxi|)^{-N},
\]
whence 
\begin{align*}
\nu(S(\cA_{k'}) \cap [0,1]^d) &\le \nu(f)  \\ &= \lam(f) \left(
1 + \lam^{-1}(f)\sum_{\bt \ne \bzero} \hat{f}(p^{k'-k-h}\bt)\hat \nu(
-p^{k'-k-h} \bt) \right).
\end{align*}

In the calculation above, we have written $\hat \nu$ for the Fourier transform of $\nu$. However, since $\gam < 1$, the measure $\nu$ is supported in a unit box. Thus, by \eqref{coincide}, its $\bZ^d$-periodic extension has Fourier coefficient $\hat \nu(\bxi)$ at any $\bxi \in \bZ^d$.

Later, we will average over $k'$, and the rapid decay of $\hat f$ will enable us to essentially reduce consideration to $|\bt| \ll p^{h'}$.  Let us first examine $\nu$ more closely. 

Let $T(\bx) = q^{-\ell} (\bx + \bb)$,
so that
$T([0,1]^d) = \cJ$. As $\mu$ is self-similar~\cite{BP},
\[
\mu_\cJ(\cA) 
= \mu(T^{-1} (\cA \cap \cJ)).
\]
Consequently,
\[
\nu(\cA) = \mu(T^{-1} S^{-1} \cA).
\]
Writing $T^{-1} S^{-1} \bx = \gam^{-1} \bx + \bs$,
for some translation $\bs \in \bR^d$, we arrive at the following equality of Fourier transforms:
\begin{align*}
|\hat \nu(\bxi)| &=
\left| \int_{S(\cJ)} e(-\bxi \cdot \bx) \d \nu(\bx) \right| \\ &= \left| \int_{S(\cJ)}
e(-\bxi \cdot \bx) 
\d \mu(\gam^{-1} \bx + \bs) \right| \\
&=  \left| \int_{[0,1]^d} e(-\gam \bxi \cdot \by) \d \mu(\by) \right| 
= |\hat{\mu}(\gamma\bxi)|.
\end{align*}

Let us now collect the contributions from all of the pieces $\cJ$. We have
\begin{align*}
\mu(\cA_k \cap \cA_{k'})
&= \sum_\cI \mu(\cA_{k'} \cap \cI)
= \sum_{\cI,\cJ} \mu(\cA_{k'} \cap \cI \cap \cJ) \\
&= \sum_{\cI, \cJ} \mu(\cJ) \mu_\cJ(\cA_{k'} \cap \cI) 
= \sum_{\cI, \cJ} \mu(\cJ) \nu_{\cI, \cJ}(S(\cA_{k'}) \cap [0,1]^d)
\end{align*}
and, for each $\cI$, there are $O_\mu(1)$ many $\cJ$. We have $$\sum_{\cJ:\cI\cap\cJ\neq\emptyset}\mu(\cJ)\ll \mu(3\cI),$$ where $3\cI$ is $\cI$ dilated by a factor of three while fixing the centre. Hence summing over $\cI,\cJ$ with $\cI\cap \cJ\neq \emptyset$, we obtain
\[
\sum_{\cI, \cJ} \mu(\cJ)\ll \mu(3\cA_k),
\]
where $3\cA_k$ is the union of the 3-dilates of the boxes comprising $\cA_k$. We write $\nu = \nu_{\cI, \cJ}$, and study 
\[
|\hat \nu(\bxi)|
= |\hat \nu(p^{k'-(k+h)}\bt)| = |\hat \mu(q^{-\ell} p^{k'}\bt)|.
\]
Notice that the $q$-expansion of $q^{-\ell} p^{k'}\bt$ is the same as that of $p^{k'}\bt$, but shifted by $\ell$ many positions (with a radix point --- the analogue of a decimal point). This motivates the use of Theorem \ref{AvgDecay2}.

Let $N$ be a positive constant, to be specified shortly.
For any $k \in \bN$,
\begin{align*}
&\sum_{\substack{Q < k' \le 2Q \\
k'-k\geq \Delta}} \mu(\cA_{k}\cap \cA_{k'}) \ll \sum_{k'}
\sum_{\cI, \cJ} \mu(\cJ) \nu_{\cI, \cJ}(S(\cA_{k'}))\\
&\ll \mu(3\cA_k) \sum_{k'} \lambda(\cA_{k'}) \left( 1 + 
\sum_{u=0}^\infty
2^{-u N}
\sum_{\lfloor 2^{u-1}\rfloor <  p^{-h'} |\bt| \le 2^u} 
|\hat{\mu} (q^{-\ell} p^{k'}\bt)|
\right).
\end{align*}
Since $\Delta \log p - \log q > Q^{(c/\log\log Q)^3}$ for large $Q$, an application of Theorem \ref{AvgDecay2} with $\tau=3$ furnishes
\[
2^{-uN}
\sum_{\lfloor 2^{u-1} \rfloor <  p^{-h'} |\bt| \le 2^u} \ \sum_{\substack{Q < k' \le 2Q \\
k'-k\geq \Delta}}
|\hat{\mu} (q^{-\ell} p^{k'}\bt)|
\ll 2^{(d-N)u} p^{dh'} Q^{-(c/\log \log Q)^4},
\]
when $2^u \le p^{-h'} Q^2$.
Choosing $N = 2(d+1)$, we thus have
\[
\sum_{\substack{u=0\\ 2^u \le p^{-h'} Q^2}}^\infty
2^{-uN}
\sum_{\lfloor 2^{u-1} \rfloor <  p^{-h'} |\bt| \le 2^u} \ \sum_{\substack{Q < k' \le 2Q \\
k'-k\geq \Delta}}
|\hat{\mu} (q^{-\ell} p^{k'}\bt)|
\ll Q (\log Q)^2 Q^{-(c/\log \log Q)^4}.
\]
Coupling this with the calculation
\[
\sum_{2^u > p^{-h'}Q^2} 2^{-uN} \sum_{|\bt| \le p^{h'} 2^u} \ \sum_{k' \le 2Q} |\hat \mu(q^{-\ell} p^{k'} \bt)| \ll (p^{h'})^{N} Q^{1+2(d-N)} \ll 1,
\]
we find
\begin{equation}
\label{LargeDifference}
\sum_{\substack{Q < k' \le 2Q \\
k'-k \geq \Delta}} \mu(\cA_{k}\cap \cA_{k'}) \ll Q \mu(3\cA_k) \lam(\cA_{Q + 1}).
\end{equation}

\subsection{Quasi-independence on average: reduction to Lebesgue case} \label{sec: reduce to Leb}

We bound, for two powers of four $Q'\geq Q>0$,
\[
\sum_{Q < k, k' \le 2Q}
\mu(\cA_{k} \cap \cA_{k'}).
\]

\bigskip
\textbf{Case: $Q=Q'$.}
By \eqref{CloseBound} and \eqref{LargeDifference},
\begin{align*}
&\sum_{Q < k, k' \le 2Q} \mu(\cA_{k} \cap \cA_{k'}) \\
&\ll \sum_{Q < k, k' \le 2Q} \lam(2\cA_{k} \cap 2\cA_{k'}) +  Q^2 \lam(\cA_{Q+1})^2 + \sum_{Q < k \le 2Q} \mu(3\cA_k) Q \lam(\cA_{Q+1}).
\end{align*}
We now use \eqref{eqn: ETP bound} to deduce that
\begin{equation*}
\sum_{Q < k,k' \le 2Q} \mu(\cA_{k} \cap \cA_{k'}) \ll \sum_{Q < k,k' \le 2Q} \lam(2\cA_{k} \cap 2\cA_{k'}) + Q^2 \lam(\cA_{Q+1})^2.
\end{equation*}

\bigskip
\textbf{Case: $4Q \leq Q'$.}
Note that if
\[
Q\leq k\leq 2Q,
\qquad
Q' \leq k' \leq 2Q',
\]
then $k'-k \ge Q'/2$. Thus, by \eqref{LargeDifference} and \eqref{eqn: ETP bound}, 
\begin{align}
\notag
&\sum_{Q < k \le 2Q} \
\sum_{Q' < k' \le 2Q'} \mu(\cA_{k} \cap \cA_{k'}) \\ 
\label{LebRed}
&\ll
\sum_{Q < k \le 2Q} \
\sum_{Q' < k' \le 2Q'} \lam(2\cA_{k} \cap 2\cA_{k'}) + QQ' \lam(\cA_{Q+1}) \lam(\cA_{Q'+1}).
\end{align}

Observe that we have \eqref{LebRed} in both cases.

\subsection{Quasi-independence on average: Lebesgue case}
\label{sec: Leb}

The Lebesgue second moment analysis was completed for measures with mild Fourier decay, including Lebesgue measure, by Pollington--Velani--Zafeiropoulos--Zorin \cite{PVZZ2022} in one dimension, and by Tan--Zhou \cite{TZ} in higher dimensions.
Let $T$ be a large enough to ensure that $S_T \ge 2$, where
\[
S_T = \sum_{k \le T} \lam(\cA_k).
\]
By
\cite[Proposition 2]{PVZZ2022} in the case $d=1$, and by \cite[Lemma 3.9]{TZ} in the case $d\geq 2$, we have
\[
\sum_{k, k' \le T} \lam(2\cA_{k} \cap 2\cA_{k'}) \ll 
\left(
\sum_{k \le T}
\lam(2\cA_{k})
\right)^2
\ll_d S_T^2.
\]
We use \eqref{eqn: ETP bound} to replace $\lam(\cA_k)$ in by $\mu(\cA_k)$ in the sum $S_T$, delivering
\begin{equation}
\label{LebQIA}
\sum_{k, k' \le T} \lam(2\cA_{k} \cap 2\cA_{k'}) \ll \left( \sum_{k \le T} \mu(\cA_k) \right)^2.
\end{equation}

\subsection{Quasi-independence on average: conclusion}

We make the following assumptions:

\begin{enumerate}
\item $\cA_k$ is empty when $2^u < k \le 2^{u+1}$ and $u$ is odd.
\item $\lam(\cA_k) = \lam(\cA_{Q + 1})$ if $Q$ is a power of four and $Q < k \le 2Q$.
\item $\sum_{k=1}^\infty \lam(\cA_k) = \infty$.
\end{enumerate}
Now, by \eqref{LebRed} and \eqref{LebQIA},
we have \eqref{QIA}.

\subsection{Proof of Theorem \ref{MainThm}, divergent part}

Let $\psi: \bN \to [0,1]$ be a non-increasing function such that 
\[
\sum_{k \ge 1}
\psi(k)^d =\infty.
\]
We first enlarge $\psi$, in the same way as for the convergent case, defining
\[
\psi^*(k) = \max \left \{ \psi(k), \frac{1}{k^{1/d} \log^{2/d} k}\right\}.
\]
Using the convergent part of Theorem \ref{MainThm}, which we have already established, we note that
$\mu(W_p(\psi^*, d)) = \mu(W_p(\psi, d))$. 

Next, we modify each value of $\psi^*$ to ensure that it is $p$-rational. This will possibly shrink each value of $\psi$ by a bounded factor. 

Finally, we use the approximation function
\[
\psi^\dag(k) = 
\begin{cases}
\psi^*(2Q), &\text{if } Q \text{ is a power of 4 and }
Q < k \le 2Q \\
0, &\text{otherwise}.
\end{cases}
\]
Having changed some of the values to zero, the approximation function is no longer monotonic. This is not a problem, because our analysis runs over dyadic ranges, and the cited results from \cite{PVZZ2022, TZ} do not assume monotonicity. In view of \eqref{DivReq} and \eqref{QIA}, the divergence Borel--Cantelli lemma (Lemma \ref{DBC})
furnishes
$\mu(W_p(\psi^\dag,d)) > 0$. In particular,
\[
\mu(W_p(\psi,d)) =
\mu(W_p(\psi^*,d)) \geq \mu(W_p(\psi^\dag,d)) > 0.
\]
Finally, by Lemma \ref{ZeroFull}, we have $\mu(W_p(\psi,d)) = 1$.

\section{Quantitative normality}\label{sec: QN}

This section will culminate in a proof of Theorem \ref{thm: QN_intro}. We will, \emph{a fortiori}, establish a more general result. Our key input is Theorem \ref{AvgDecay2} and the idea of its proof. The rest of the argument is quite standard in the field of ergodic theory, but we provide details for completeness. We prove the following result covering Theorem \ref{thm: QN_intro}.

\begin{thm}
\label{thm: QN}
Let $f\in C(\bT)$ with $$\hat{f}(\xi) \ll_{f,\sig} |\xi|^{-\sig}$$ for some $\sig>1$, e.g. $f\in C^2(\bT)$.
Let $\mu$ be a missing-digit measure with prime base $q$, and let $p \ge 2$ be an integer such that $q \nmid p$. Then for any $\delta>0$, for $\mu$-a.e $x$,
\[
\left| \frac{1}{T} \sum_{k=0}^{T-1}f(p^k x) -\int_{\bT} f(x) \d x \right| \ll_{f,\sig, \del,x} T^{\del - 1/2}
\qquad (T \ge 1).
\]
\end{thm}

\begin{rem}
In principle, further refinements would handle the endpoint case $\sigma=1$, thereby permitting any $f$ of bounded variation. We do not pursue this level of precision, since we do not believe bounded variation to be the optimal condition anyway. Perhaps some H\"{o}lder condition is sufficient.
\end{rem}

\begin{proof}
We follow the standard $L^2$ approach. Let
\[
A^f_T(x) = \frac1T \sum_{k=0}^{T-1} f(p^k x).
\]
Let $\bE f = \int_\bT f(x) \d x$.
We may assume that $\bE f = 0$. We further define, with the zero-mean assumption,
\[
V^f_{T,\max} = \int \max_{1\leq n\leq T}|nA^f_n|^2 \d\mu(x).
\]
We will see in the next lemma that $V^f_{T,\max} \ll_{f,\del} T^{1 + \del}$ for any $\delta > 0.$ Using this with Markov's inequality and the first Borel--Cantelli lemma, we see that for any $\delta>0$, and $\mu$-a.e. $x$, 
\[
\max_{1\leq n\leq T}|nA^f_n(x)| \leq T^{(1+2\delta)/2}
\]
holds for all sufficiently large $T=2^\ell$, where $\ell\in\bN.$ For such $x,$
\[
\frac1T \sum_{k=0}^{T-1} f(p^k x) \ll T^{\del - 1/2}
\qquad (T \ge 1).
\]
\end{proof}

It remains to establish the lemma that was used in the proof above.

\begin{lem}
Let $f\in C(\bT)$ with $\bE f=0$ and
\[
\hat{f}(\xi) \ll_{f,\sig} |\xi|^{-\sig},
\]
for some $\sig>1$. Let $\mu$ be a missing-digit measure with prime base $q$,
and let $p \ge 2$ be an integer such that $q \nmid p$. Let $T = 2^\ell$, where $\ell \in \bN$, and define
\[
V^f_{T,\max}
=
\int_\bT
\max_{1\leq n\leq T}|nA^f_n(x)|^2
\,\d\mu(x).
\]
Then for any $\del>0$,
\[
V^f_{T,\max}\ll_{f,\sig,\del}T^{1+\del}
\qquad(T\geq1).
\]
\end{lem}

\begin{proof}
We first suppose that $T=2^\ell$ for some $\ell\in\bN$. For $\xi\neq0$, put
\[
Y^\xi_k(x)=e^{-2\pi i p^k x\xi}.
\]
Fix $\varpi>0$, to be chosen later. We first consider
\[
1\leq|\xi|\leq T^\varpi.
\]
For every $\eta>0$ 
we claim that,
uniformly for
$1\leq h<T$,
\begin{equation}
\label{eqn: uniform maximal Fourier}
\sum_{k=0}^{T-1}
\left|
\hat{\mu}\bigl(p^k\xi(p^h-1)\bigr)
\right|
\ll_{\varpi,\eta} T^\eta.
\end{equation}
Indeed, the terms $k\leq T^\eta$ are bounded trivially. For the remaining
terms we decompose $k$ dyadically. If $Q<k\leq2Q$ and $Q\geq T^\eta/2$,
then
\[
\log T\ll_\eta\log Q.
\]
The base-$p$ expansion of $\xi(p^h-1)$ consists of at most two blocks of
$O_{\varpi,\eta}(\log Q)$ many digits, separated by a string of digits $p-1$.
We apply Lemma \ref{lem: two_blocks}, and summing the dyadic estimates furnishes
\eqref{eqn: uniform maximal Fourier}.

We now prove a maximal estimate for the exponential sums. Put
\[
G_T(\xi)
=
\sum_{0\leq k,k'<T}
\left|
\int_\bT
Y^\xi_k(x)\overline{Y^\xi_{k'}(x)}
\,\d\mu(x)
\right|.
\]
Since
\[
\int_\bT
Y^\xi_k(x)\overline{Y^\xi_{k+h}(x)}
\,\d\mu(x)
=
\hat{\mu}\bigl(p^k\xi(1-p^h)\bigr),
\]
we obtain from \eqref{eqn: uniform maximal Fourier}
\begin{align}
G_T(\xi)
=
T
+
2\sum_{h=1}^{T-1} \ 
\sum_{k=0}^{T-h-1}
\left|
\hat{\mu}\bigl(p^k\xi(1-p^h)\bigr)
\right| \ll_{\varpi,\eta}
T^{1+\eta}.
\label{eqn: maximal correlation sum}
\end{align}

For $j=0,1,\ldots,\ell$, let $\cP_j$ be the partition of
$[0,T)\cap\bZ$ into intervals of length $2^j$.
Each initial interval $[0,n)\cap\bZ$, with $1\leq n\leq T$, is a disjoint union of at most $\ell+1$ many dyadic intervals, using at most one interval
at each scale:
\[
[0,n)\cap\bZ = \bigsqcup_{I \in \mathcal D(n)} I, \qquad
\#\mathcal D(n)\leq \ell+1.
\]
Hence, by Cauchy--Schwarz, if $1 \le n \le T$ then
\begin{align*}
\left|
\sum_{k=0}^{n-1}Y^\xi_k(x)
\right|^2
&=
\left|
\sum_{I\in\mathcal D(n)}
\sum_{k\in I}Y^\xi_k(x)
\right|^2 \leq
(\ell+1)
\sum_{I\in\mathcal D(n)}
\left|
\sum_{k\in I}Y^\xi_k(x)
\right|^2\\
&\leq
(\ell+1)
\sum_{j=0}^{\ell}
\sum_{I\in\cP_j}
\left|
\sum_{k\in I}Y^\xi_k(x)
\right|^2.
\end{align*}
Thus,
\[
\max_{1\leq n\leq T}
\left|
\sum_{k=0}^{n-1} Y^\xi_k(x)
\right|^2
\leq
(\ell+1)
\sum_{j=0}^{\ell}
\sum_{I\in\cP_j}
\left|
\sum_{k\in I} Y^\xi_k(x)
\right|^2.
\]

Integrating against $\mu$, and using the triangle inequality at each fixed
scale, we find
\begin{align*}
&\int_\bT
\max_{1 \leq n \leq T}
\left|
\sum_{k=0}^{n-1} Y^\xi_k(x)
\right|^2
\, \d\mu(x) \\
&\leq
(\ell+1)
\sum_{j=0}^{\ell} \ 
\sum_{I\in\cP_j} \ 
\sum_{k,k'\in I}
\left|
\int_\bT
Y^\xi_k(x)\overline{Y^\xi_{k'}(x)}
\,\d\mu(x)
\right| \leq
(\ell+1)^2 G_T(\xi).
\end{align*}
The last inequality follows because, for any $j$,
\[
\sum_{I \in \cP_j} \
\sum_{k,k'\in I}
\left|
\int_\bT
Y^\xi_k(x) \overline{Y^\xi_{k'}(x)} \, 
\d\mu(x)
\right|\leq G_T(\xi).
\]
Together with \eqref{eqn: maximal correlation sum}, it yields
\begin{equation}
\label{eqn: single frequency maximal}
\left\|
\max_{1\leq n\leq T}
\left|
\sum_{k=0}^{n-1}
e^{-2\pi i p^kx\xi}
\right|
\right\|_{L^2(\mu)}
\ll_{\varpi,\eta}
T^{(1+\eta)/2}(1+\log T),
\end{equation}
uniformly for $1\leq|\xi|\leq T^\varpi$.

We now return to $f$. Since $\bE f=0$ and $\sig>1$, its Fourier series is
absolutely convergent:
\[
f(x)=
\sum_{\xi\neq0}
\hat f(\xi) e^{2\pi i\xi x},
\qquad
\sum_{\xi\neq0}|\hat f(\xi)|<\infty.
\]
Choose the constant $\varpi>0$ sufficiently large to ensure that
\[
\varpi(\sig-1)>1.
\]
For every $1\leq n\leq T$,
\begin{align*}
|nA^f_n(x)|
\leq
\sum_{0<|\xi|\leq T^\varpi}
|\hat f(\xi)|
\left|
\sum_{k=0}^{n-1}
e^{-2\pi i p^kx\xi}
\right| +
n\sum_{|\xi|>T^\varpi}|\hat f(\xi)|.
\end{align*}
Since
\[
\sum_{|\xi|>T^\varpi}| \hat f(\xi)|
\ll_{f, \sig} T^{-\varpi(\sig-1)},
\]
Minkowski's inequality and
\eqref{eqn: single frequency maximal} give
\begin{align*}
\left\|
\max_{1\leq n\leq T}|nA^f_n|
\right\|_{L^2(\mu)}
&\ll_{f,\sig,\eta}
T^{(1+\eta)/2}(1+\log T)
\sum_{\xi\neq0}|\hat f(\xi)| +
T^{1-\varpi(\sig-1)}\\
&\ll 
T^{(1+\eta)/2}(1+\log T),
\end{align*}
whence,
\[
V^f_{T,\max}
\ll_{f,\sig,\eta}
T^{1+\eta}(1+\log T)^2.
\]
Taking, for example, $\eta=\del/2$, and absorbing the logarithmic factor
into $T^{\del/2}$ for large $T$, we obtain
$
V^f_{T,\max}\ll_{f,\sig,\del}T^{1+\del}.
$

Finally, if $T$ is not a power of two, then let $T'$ be the unique power of
two for which $T\leq T'<2T$. Then
\[
V^f_{T,\max}\leq V^f_{T',\max}
\ll (T')^{1+\del}
\ll T^{1+\del}.
\]
\end{proof}

\section{Experimental data}
\label{sec: Exp}

Here, we take an empirical approach to the question posed in \S \ref{ques: CLT}. We choose $p=2$ and $q=3$.

We begin by generating a random $0,1$ sequence of  length $2^{28}$, which is roughly $33.55$MB, on a 202X LINUX machine. 
Let $L,R \in \bN$ with \mbox{$LR\leq 2^{28}$.} We split the random $0,1$ sequence into $R$ many sequences of length $L$. For each such sequence $\omega$, we construct a number in the middle-third Cantor set by forming a ternary expansion with $0$ corresponding the ternary digit $0$, and $1$ corresponding to the ternary digit $2$. 
We regard $x_\omega$ as a random variable whose law is the standard Cantor--Lebesgue measure. We may then evaluate $A^f_T(x_\omega)$ with a suitable choice of $T$, where $f \in C(\bT)$. Then we compute $Z^f_T(x_\omega)$ and make a histogram of $Z^f_T(x_\omega)$ over the $R$ many $\omega$. 

Suppose $TL \leq 2^{28}$. We repeat the procedure above, but this time, we view each $\omega$ as the binary expansion of a dyadic rational in $[0,1)$. The number $x^B_\omega$ can be viewed as a Lebesgue random number. We can compute $Z^f_T(x^B_\omega)$ and study its distribution.

For concreteness, we choose $T=2048,L=2112,R=16384$, $f(x)=\cos(2\pi x)$. We numerically  evaluate $f(x_\omega),\dots,f(2^{T-1}x_\omega)$. When $x_\omega$ is obtained via ternary expansion, we must be sure that the fractional part of $2^{T-1}x_\omega$ still has enough `effective' ($\approx$ aperiodic) digits. 
For a number with $2112$ ternary digits, its binary expansion has over $3000$ effective binary digits. Here the numbers are only heuristics. The intuition is that $2112$ randomly-chosen ternary $0,2$ digits result in over $3000$ seemingly random binary digits. The choice $T=2048$ is much smaller than $3000$, so we should expect $x_\omega, \dots,2^{T-1} x_\omega \mod 1$ to behave randomly. 

\begin{figure}[H]
\includegraphics[width=0.9\textwidth]{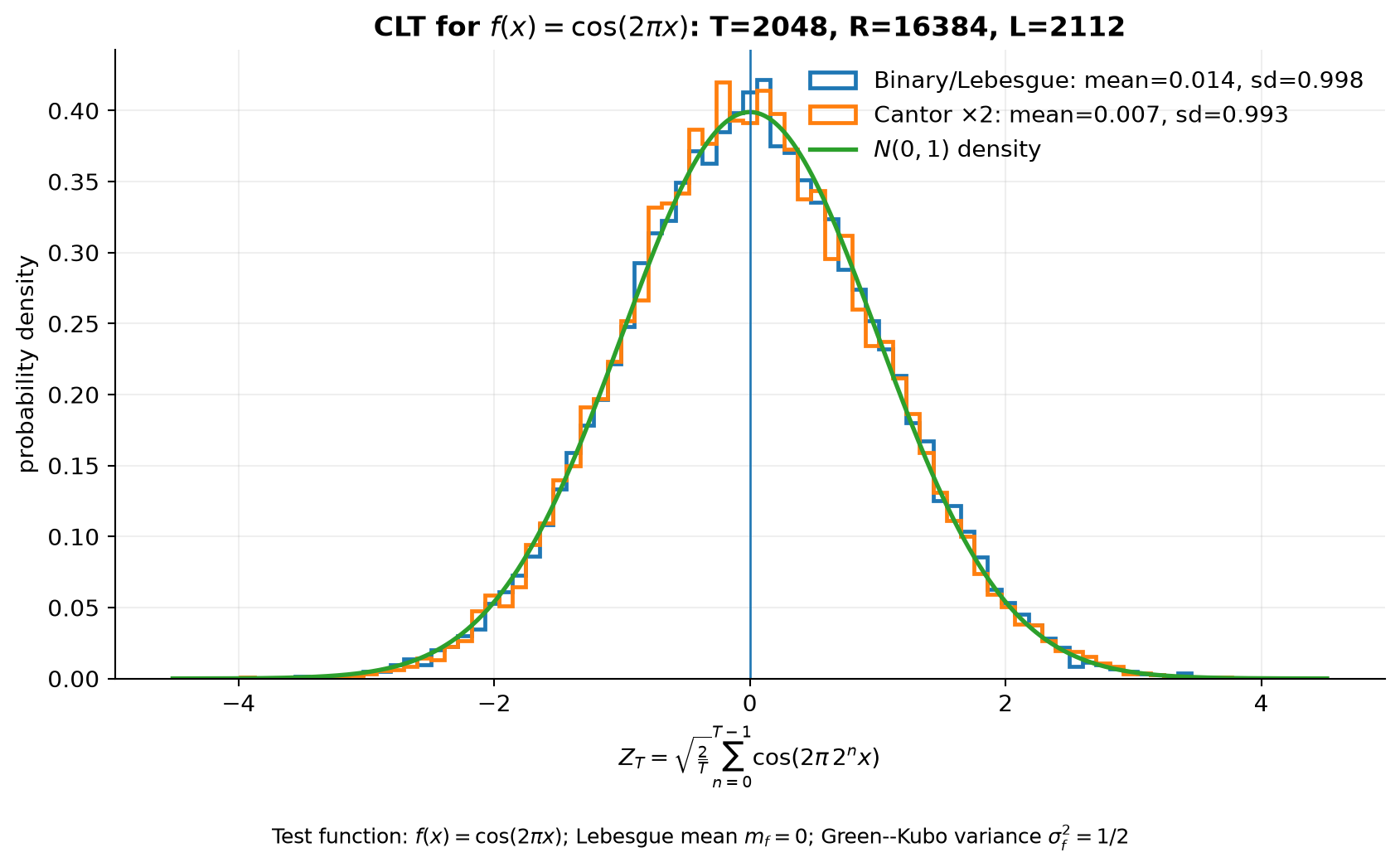}
\caption{CLT plot: the orange curve is the distribution of $x_\omega$ as mentioned above; the blue line is the Lebesgue benchmark.}
\label{GoodFigure}
\end{figure}

This is the aforementioned strong numerical evidence towards an affirmative answer to the question posed in \S \ref{ques: CLT} with $p=2$, $q=3$, $f(x) = \cos(2\pi x)$, and $\mu$ being the Cantor--Lebesgue measure on the middle-third Cantor set. This in turn would show that, in this specific setting, the exponent $1/2$ in Theorem \ref{thm: QN_intro} is optimal. 

\bibliographystyle{amsplain}

\begin{thebibliography}{99}


\bibitem{Alg} A. Algom, \emph{Recent progress on pointwise normality of self-similar measures}. In: Fractal Geometry and Stochastics VII, 2026. Preprint available at arXiv:2504.18192.

\bibitem{ABCY2023}
D. Allen, S. Baker, S. Chow and H. Yu, \emph{A note on dyadic approximation in Cantor's set}, Indag. Math. (N. S.) \textbf{34} (2023), 190--197.

\bibitem{ACY2023} D. Allen, S. Chow and H. Yu, \emph{Dyadic Approximation in the Middle-Third Cantor Set}, Selecta Math. (N. S.) \textbf{29} (2023), Paper No. 11, 49 pages.

\bibitem{Bak2025}
S. Baker, \emph{Approximating elements of the middle third Cantor set with dyadic rationals}, Israel J. Math. \textbf{266} (2025), 285--305.

\bibitem{Ban}
P. Bandi, \emph{Averaged Fourier estimates and dyadic approximation on the Cantor set}, arXiv:2606.27034.

\bibitem{BHZ2026}
T. B\'enard, W. He and H. Zhang, \emph{Khintchine dichotomy for self-similar measures}, J. Amer. Math. Soc. \textbf{39} (2026),  587--623.

\bibitem{BHZ}
T. B\'enard, W. He and H. Zhang, \emph{Khintchine dichotomy and Schmidt estimates for self-similar measures on $\bR^d$}, arXiv:2508.09076.

\bibitem{BDV2006}
V. Beresnevich, 
D. Dickinson and S. Velani, \emph{Measure theoretic laws for lim sup sets}, Mem. Amer. Math. Soc. \textbf{179} (2006).

\bibitem{BV2023}
V. Beresnevich and S. Velani, \emph{The divergence Borel-Cantelli Lemma revisited}, J. Math. Anal. Appl. \textbf{519} (2023), Paper No. 126750, 21 pp.

\bibitem{BP} C. Bishop and Y. Peres, \emph{Fractals in Probability and Analysis}, Cambridge Stud. Adv. Math. \textbf{162,} Cambridge University Press, Cambridge, 2017.


\bibitem{Bre21} J. Br\'{e}mont, \emph{Self-similar measures and the Rajchman property}, Ann. H. Lebesgue \textbf{4} (2021), 973--1004.



\bibitem{Cas1950}
J. W. S. Cassels, \emph{Some metrical theorems in Diophantine approximation. I}, Proc. Cambridge Philos. Soc. \textbf{46} (1950), 209--218.

\bibitem{Cas1959} J. W. S. Cassels, \emph{On a problem of Steinhaus about normal numbers}, Colloq. Math. \textbf{7} (1959), 95--101.

\bibitem{CVY2026} S. Chow, P. P. Varj\'{u} and H. Yu, \emph{Counting rationals and diophantine approximation in missing-digit Cantor sets}, Adv. Math. \textbf{488} (2026), 30 pages.

\bibitem{CY1}
S. Chow and H. Yu,
\emph{Moment transference principles and multiplicative diophantine approximation on hypersurfaces}, arXiv:2408.10911.

\bibitem{CY2}
S. Chow and H. Yu, \emph{Simultaneous and multiplicative Diophantine approximation on missing-digit fractals}, arXiv:2412.12070.

\bibitem{DLWW}
X.-R. Dai, B. Li, B. Wang and Y.-F. Wu, \emph{Metric results for dyadic approximation on the middle-third Cantor set}, arXiv:2606.25305.

\bibitem{DJ}
S. Datta and S. Jana, \emph{On Fourier Asymptotics and Effective Equidistribution}, J. Reine Angew. Math., to appear.

\bibitem{DGW24} Y. Dayan, A. Ganguly and B. Weiss, \emph{Random walks on tori and normal numbers in self-similar sets}, Amer. J. Math. \textbf{14} (2024), 467--493.

\bibitem{DS41} R. Duffin and A. Schaeffer, \emph{Khintchine's problem in metric diophantine approximation}. Duke Math. J. \textbf{8} (1941), 243--255.


\bibitem{Erd1940}
P. Erd\H{o}s, \emph{On the smoothness properties of Bernoulli convolutions}, Amer. J. Math. \textbf{62} (1940),
180--186.



\bibitem{Fur1967}
H. Furstenberg, \emph{Disjointness in ergodic theory, minimal sets, and a problem in diophantine approximation}, Math. Syst. Theory \textbf{1} (1967), 1--49.

\bibitem{Fur1970}
H. Furstenberg, \emph{Intersections of Cantor Sets and Transversality of Semigroups}, Problems in Analysis, Princeton University Press 1970, 41--59.



\bibitem{Gal1961}
P. X. Gallagher, \emph{Approximation by reduced fractions}, J. Math. Soc. Japan \textbf{13} (1961), 342--345.

\bibitem{Gra2014}
L. Grafakos, \emph{Classical Fourier Analysis}, third edition, Graduate Texts in Mathematics \textbf{249,} Springer, New York, 2014.








\bibitem{Kac46} M. Kac, \emph{On the distribution of values of sums of the type $\sum f(2^k t)$}, Ann. of Math. (2) \textbf{47} (1946), 33--49.

\bibitem{Kah1971}
J. P. Kahane, \emph{Sur la distribution de certaines series aleatoires}, Colloque Th. Nombres [1969, Bordeaux], Bull. Soc. Math. France, Memoire \textbf{25} (1971), 119--122.






\bibitem{KL2023} O. Khalil and M. Luethi, \emph{Random Walks, Spectral Gaps, and Khintchine's Theorem on Fractals}, Invent. Math. \textbf{232} (2023), pages 713--831.

\bibitem{Khi1924} A. Khintchine, \emph{Einige S\"{a}tze \"{u}ber Kettenbr\"{u}che, mit Anwendungen auf die Theorie der Diophantischen Approximationen}, Math. Ann. \textbf{92} (1924), 115--125.


\bibitem{KM20} D. Koukoulopoulos and J. Maynard, \emph{On the Duffin-Schaeffer conjecture}, Ann. of Math. (2) \textbf{192} (2020), 251--257.



\bibitem{LSV2007} J. Levesley, C. Salp and S. Velani,
\emph{On a problem of K. Mahler: Diophantine approximation and Cantor sets}, Math. Ann. \textbf{338} (2007), 97--118.


\bibitem{Mah1984} K. Mahler,
\emph{Some suggestions for further research}, Bull. Austral. Math. Soc. \textbf{29} (1984), 101--108.













\bibitem{Phi75} W. Philipp, \emph{Limit theorems for lacunary series and uniform distribution $\mmod 1$}, Acta Arith. \textbf{26} (1975), 241--251.

\bibitem{PVZZ2022}
A. D. Pollington, S. Velani, A. Zafeiropoulos and E. Zorin, \emph{Inhomogeneous
diophantine approximation on $M_0$ sets with restricted denominators}, Int. Math.
Res. Not. \textbf{2022,} 8571--8643.


\bibitem{Rap22} A. Rapaport, \emph{On the Rajchman property for self-similar measures on $\bR^d$}, Adv. Math. \textbf{403} (2022).






\bibitem{Ste1980}
C. L. Stewart, \emph{On the representation of an integer in two different bases}, J. Reine Angew. Math. \textbf{319} (1980),
63--72.

\bibitem{TZ} B. Tan and Q. Zhou, \emph{Quantitative Matrix-Driven Diophantine approximation on $M_0$-sets}, arXiv:2504.21555.







\bibitem{VY22} P. P. Varj\'{u} and H. Yu, \emph{Fourier decay of self-similar measures and self-similar sets of uniqueness}, Anal. PDE \textbf{15} (2022), 843--858.









\bibitem{Yu} H. Yu, \emph{Rational points near self-similar sets}, arXiv:2101.05910.



\end{thebibliography}

\end{document}